\documentclass[reqno]{amsart}
\usepackage{amsthm, amsmath, amsfonts, amssymb}
\usepackage{enumerate, mathabx}

\usepackage{color}

\newcommand{\R}[0]{\mathbb{R}}

\newcommand{\C}[0]{\mathbb{C}}
\newcommand{\N}[0]{\mathbb{N}}

\newcommand{\ta}[0]{\theta}
\newcommand{\bs}[0]{\setminus}
\newcommand{\ld}[0]{\lambda}
\newcommand{\vep}[0]{\varepsilon}
\newcommand{\supp}[0]{\operatorname{supp}}
\newcommand{\lsm}[0]{\lesssim}

\newcommand{\om}[0]{\omega}

\newcommand{\vp}[0]{\varphi}

\newcommand{\wh}[1]{\widehat{#1}}
\newcommand{\wc}[1]{\widecheck{#1}}
\newcommand{\mc}[1]{\mathcal{#1}}
\newcommand{\ov}[1]{\overline{#1}}
\newcommand{\wt}[1]{\widetilde{#1}}
\newcommand{\st}[1]{\substack{#1}}
\newcommand{\mb}[1]{\mathbf{#1}}

\newcommand{\nms}[1]{\| #1 \|}

\newcommand{\geom}[0]{\operatorname{geom}}

\newcommand{\E}[0]{\mathcal{E}}

\newcommand{\avg}[1]{\underset{#1}{\operatorname{Avg}}}

\newcommand{\Mod}[1]{\ (\mathrm{mod}\ #1)}

\newtheorem{thm}{Theorem}[section]
\newtheorem{lemma}[thm]{Lemma}

\newtheorem{cor}[thm]{Corollary}
\newtheorem{defn}[thm]{Definition}

\theoremstyle{remark}
\newtheorem{rem}{Remark}

\keywords{$l^2$ decoupling, efficient congruencing}
\subjclass[2010]{Primary 42B15; Secondary 11L07}

\title[An $l^2$ decoupling interpretation of efficient congruencing]{An $l^2$ decoupling interpretation of efficient congruencing: the parabola}
\author{Zane Kun Li}
\address{Department of Mathematics,
Indiana University\\
831 East 3rd St\\
Bloomington IN 47405\\
USA}
\email{zkli@iu.edu}

\begin{document}
\begin{abstract}
We give a new proof of $l^2$ decoupling for the parabola inspired from
efficient congruencing. Making quantitative this proof matches
a bound obtained by Bourgain for the discrete restriction problem for the parabola.
We illustrate similarities and differences
between this new proof and efficient congruencing and the proof of decoupling by Bourgain and Demeter.
We also show
where tools from decoupling such as $l^2 L^2$ decoupling, Bernstein's inequality, and ball inflation come into play.
\end{abstract}

\maketitle

\section{Introduction}
For an interval $J \subset [0, 1]$ and $g: [0, 1] \rightarrow \C$, we define
\begin{align*}
(\E_{J}g)(x) := \int_{J}g(\xi)e(\xi x_1 + \xi^2 x_2)\, d\xi
\end{align*}
where $e(a) := e^{2\pi i a}$.
For an interval $I$, let $P_{\ell}(I)$ be the partition
of $I$ into intervals of length $\ell$. By writing $P_{\ell}(I)$, we are assuming that $|I|/\ell \in \N$.
We will also similarly define $P_{\ell}(B)$ for squares $B$ in $\R^2$.
Next if $B = B(c, R)$ is a square in $\R^2$ centered at $c$ of side length $R$,
let
\begin{align}\label{wbdef}
w_{B}(x) := (1 + \frac{|x - c|}{R})^{-100}.
\end{align}
We will always assume that our squares have sides parallel to the $x$ and $y$-axis.
We observe that $1_B \leq 2^{100}w_B$.
For a function $w$, we define $$\nms{f}_{L^{p}(w)} := (\int_{\R^2}|f(x)|^{p}w(x)\, dx)^{1/p}.$$

For $\delta \in \N^{-1} = \{n^{-1} : n \in \N\}$, let $D(\delta)$ be the best constant such that
\begin{align}\label{decdef}
\nms{\E_{[0, 1]}g}_{L^{6}(B)} \leq D(\delta)(\sum_{J \in P_{\delta}([0, 1])}\nms{\E_{J}g}_{L^{6}(w_B)}^{2})^{1/2}
\end{align}
for all $g: [0, 1] \rightarrow \C$ and all squares $B$ in $\R^2$ of side length $\delta^{-2}$.
Let $D_{p}(\delta)$ be the decoupling constant where the $L^6$ in \eqref{decdef} is replaced with $L^{p}$.
Since $1_B \lsm w_B$, the triangle inequality combined with the Cauchy-Schwarz inequality shows that $D_{p}(\delta) \lsm_{p} \delta^{-1/2}$ for all $1 \leq p \leq \infty$.
The $l^2$ decoupling theorem for the paraboloid proven by Bourgain and Demeter in \cite{bd} implies that for the parabola
we have $D_{p}(\delta) \lsm_{\vep} \delta^{-\vep}$ for $2 \leq p \leq 6$ and this range of $p$ is sharp.

Decoupling-type inequalities were first studied by Wolff in \cite{wolff}.
Following the proof of $l^2$ decoupling for the paraboloid by Bourgain and Demeter
in \cite{bd}, decoupling inequalities for various curves and surfaces have found many applications to
analytic number theory (see for example \cite{zeta, bourgainmvt, bdweyl, bdguo, bdg, meansquare, prend, guozhang, guozorin, heathbrown}).
Most notably is the proof of Vinogradov's mean value theorem by Bourgain-Demeter-Guth using decoupling for the moment curve $t \mapsto (t, t^2, \ldots, t^n)$
in \cite{bdg}. Wooley in \cite{nested} was also able to prove Vinogradov's mean value theorem using his nested efficient congruencing method.

This paper probes the connections between efficient congruencing and $l^2$ decoupling in the simplest case of the parabola.
For a slightly different interpretation of the relation between efficient congruencing and decoupling for the cubic moment curve inspired
from \cite{hbwooley}, see \cite{guoliyung}.
See also \cite{glyzk} for an interpretation of \cite{nested} in the decoupling language which provides an alternative proof of decoupling
for the moment curve in $\R^d$ different from the proof in \cite{bdg}.

Our proof of $l^2$ decoupling for the parabola is inspired by the
exposition of Wooley's efficient congruencing in Pierce's Bourbaki seminar exposition \cite[Section 4]{pierce}. This proof will give the following result.
\begin{thm}\label{ef2d_main}
For $\delta \in \N^{-1}$ such that $0 < \delta < e^{-200^{200}}$, we have
\begin{align*}
D(\delta) \leq \exp(30\frac{\log\frac{1}{\delta}}{\log\log\frac{1}{\delta}}).
\end{align*}
\end{thm}
In the context of discrete Fourier restriction, Theorem \ref{ef2d_main} implies that
for all $N$ sufficiently large and arbitrary sequence $\{a_n\} \subset l^2$, we have
\begin{align*}
\nms{\sum_{|n| \leq N}a_n e^{2\pi i (nx + n^2 t)}}_{L^{6}(\mathbb{T}^2)} \lsm \exp(O(\frac{\log N}{\log\log N}))(\sum_{|n| \leq N}|a_n|^{2})^{1/2}
\end{align*}
which rederives (up to constants) the upper bound obtained by Bourgain in \cite[Proposition 2.36]{bourgain}
but without resorting to use of a divisor bound.
It is an open problem whether the $\exp(O(\frac{\log N}{\log\log N}))$ can be improved.

\subsection{More notation and weight functions}
We define
\begin{align*}
\nms{f}_{L^{p}_{\#}(B)} := (\frac{1}{|B|}\int_{B}|f(x)|^{p}\, dx)^{1/p}, \quad \nms{f}_{L^{p}_{\#}(w_B)} := (\frac{1}{|B|}\int|f|^{p}w_B)^{1/p},
\end{align*}
and given a collection $\mc{C}$ of squares, we let $$\avg{\Delta \in \mc{C}}\,f(\Delta) := \frac{1}{|\mc{C}|}\sum_{\Delta \in \mc{C}}f(\Delta).$$

Finally we will let $\eta$ be a Schwartz function such that
$\eta \geq 1_{B(0, 1)}$ and $\supp(\wh{\eta}) \subset B(0, 1)$. For $B = B(c, R)$ we also define
$\eta_{B}(x) := \eta(\frac{x - c}{R})$. In
Section \ref{fc} we care about explicit constants and so we will use the explicit $\eta$ constructed in Corollary \ref{etaweightwb}.
In particular, for this $\eta$, $\eta_{B} \leq 10^{2400}w_B$.
For the remaining sections in this paper, we will ignore this constant.
The most important facts about $w_B$ we will need are that
$$w_{B(0, R)} \ast w_{B(0, R)} \lsm R^{2}w_{B(0, R)}$$
and
$$1_{B(0, R)} \ast w_{B(0, R)} \gtrsim R^{2}w_{B(0, R)}$$
from which we can derive all the other properties about weights we will use such as given a partition $\{\Delta\}$ of $B$, $\sum_{\Delta}w_{\Delta} \lsm w_B$
and $$\nms{f}_{L^{p}(w_{B(0, R)})}^{p}\lsm \int_{\R^2}\nms{f}_{L^{p}_{\#}(B(y, R))}^{p}w_{B(0, R)}(y)\, dy.$$
We refer the reader to \cite[Section 4]{sg} and \cite[Section 2.2]{thesis} for some useful details and properties of the weights $w_B$ and $\eta_B$.
To keep the paper relatively self contained, we have also included proofs of these estimates in Section \ref{explicitdis} with explicit constants.

\subsection{Outline of proof of Theorem \ref{ef2d_main}}
Our argument is inspired by the discussion of efficient congruencing in
\cite[Section 4]{pierce} which in turn is based off Heath-Brown's simplification \cite{hbwooley} of Wooley's proof of the cubic case
of Vinogradov's mean value theorem \cite{wooleycubic}.

Our first step, much like the first step in both efficient congruencing and decoupling for the parabola, is to bilinearize the problem.
Throughout we will assume $\delta^{-1} \in \N$ and $\nu \in \N^{-1} \cap (0, 1/100)$.

Fix arbitrary integers $a, b \geq 1$.
Suppose $\delta$ and $\nu$ were such that $\nu^{a}\delta^{-1}, \nu^{b}\delta^{-1} \in \N$.
This implies that $\delta \leq \min(\nu^a, \nu^b)$ and the requirement
that $\nu^{\max(a, b)}\delta^{-1}\in \N$ is equivalent to having $\nu^{a}\delta^{-1}, \nu^{b}\delta^{-1} \in \N$.
For this $\delta$ and $\nu$, let $M_{a, b}(\delta, \nu)$ be the best constant such that
\begin{align}\label{mabdef}
\int_{B}|\E_{I}g|^{2}|\E_{I'}g|^{4} \leq M_{a, b}(\delta, \nu)^{6}(\sum_{J \in P_{\delta}(I)}\nms{\E_{J}g}_{L^{6}(w_B)}^{2})(\sum_{J' \in P_{\delta}(I')}\nms{\E_{J'}g}_{L^{6}(w_B)}^{2})^{2}
\end{align}
for all squares $B$ of side length $\delta^{-2}$, $g: [0, 1] \rightarrow \C$, and all
intervals $I \in P_{\nu^{a}}([0, 1])$, $I' \in P_{\nu^{b}}([0, 1])$ with $d(I, I') \geq 3\nu$.
We will say that such $I$ and $I'$ are $3\nu$-separated.
Applying H\"{o}lder's inequality followed by the triangle inequality and the Cauchy-Schwarz inequality shows that $M_{a, b}(\delta, \nu)$ is finite.
This is not the only bilinear decoupling constant we can use (see \eqref{bik_const} and \eqref{mb_bds} in Sections \ref{bik} and \ref{bds}, respectively),
but in this outline we will use \eqref{mabdef} because it is closest to the one used in \cite{pierce} and the one we will use in Section \ref{fc}.

Our proof of Theorem \ref{ef2d_main} is broken into the following four lemmas. We state them below ignoring explicit constants for now.

\begin{lemma}[Parabolic rescaling]\label{parab_outline}
Let $0 < \delta < \sigma < 1$ be such that $\sigma, \delta, \delta/\sigma \in \N^{-1}$. Let $I$ be an arbitrary interval in $[0, 1]$ of length $\sigma$. Then
\begin{align*}
\nms{\E_{I}g}_{L^{6}(B)} \lsm D(\frac{\delta}{\sigma})(\sum_{J \in P_{\delta}(I)}\nms{\E_{J}g}_{L^{6}(w_B)}^{2})^{1/2}
\end{align*}
for every $g: [0, 1] \rightarrow \C$ and every square $B$ of side length $\delta^{-2}$.
\end{lemma}

\begin{lemma}[Bilinear reduction]\label{bi_outline}
Suppose $\delta$ and $\nu$ were such that $\nu\delta^{-1} \in \N$. Then
$$D(\delta) \lsm D(\frac{\delta}{\nu}) + \nu^{-1}M_{1, 1}(\delta, \nu).$$
\end{lemma}

\begin{lemma}\label{freq_outline}
Let $a$ and $b$ be integers such that $1 \leq a \leq 2b$.
Suppose $\delta$ and $\nu$ were such that $\nu^{2b}\delta^{-1} \in \N$.
Then
$$M_{a, b}(\delta, \nu) \lsm \nu^{-1/6}M_{2b, b}(\delta, \nu).$$
\end{lemma}

\begin{lemma}\label{switch_outline}
Suppose $b$ is an integer and $\delta$ and $\nu$ were such that $\nu^{2b}\delta^{-1} \in \N$.
Then
$$M_{2b, b}(\delta, \nu) \lsm M_{b, 2b}(\delta, \nu)^{1/2}D(\frac{\delta}{\nu^b})^{1/2}.$$
\end{lemma}

Applying Lemma \ref{freq_outline}, we can move from
$M_{1, 1}$ to $M_{2, 1}$ and then Lemma \ref{switch_outline} allows us to move from
$M_{2, 1}$ to $M_{1, 2}$ at the cost of a square root of $D(\delta/\nu)$. Applying Lemma \ref{freq_outline} again moves us to $M_{2, 4}$.
Repeating this we can eventually reach $M_{2^{N-1}, 2^{N}}$ paying some $O(1)$ power of $\nu^{-1}$
and the value of the linear decoupling constants at various scales.
This combined with Lemma \ref{bi_outline} and the choice of $\nu = \delta^{1/2^N}$ leads to the following result.
\begin{lemma}
Let $N \in \N$ and suppose $\delta$ was such that $\delta^{-1/2^N} \in \N$ and $0 < \delta < 100^{-2^N}$.
Then
\begin{align*}
D(\delta) \lsm D(\delta^{1 - \frac{1}{2^{N}}}) + \delta^{-\frac{4}{3\cdot 2^{N}}}D(\delta^{1/2})^{\frac{1}{3\cdot 2^{N}}}\prod_{j = 0}^{N-1}D(\delta^{1 - \frac{1}{2^{N - j}}})^{\frac{1}{2^{j + 1}}}.
\end{align*}
\end{lemma}
This then gives a recursion which shows that $D(\delta) \lsm_{\vep} \delta^{-\vep}$
(see Section \ref{fc_iter} for more details).

The proof of Lemma \ref{parab_outline} is essentially a change of variables and applying
the definition of the linear decoupling constant (some small technical issues arise because of the weight
$w_B$, see \cite[Section 2.4]{thesis}). The idea is that a cap on the paraboloid can be stretched
to the whole paraboloid without changing any geometric properties.
The bilinear reduction Lemma \ref{bi_outline} follows from H\"{o}lder's inequality.
The argument we use is from Tao's exposition on the Bourgain-Demeter-Guth proof of Vinogradov's mean value theorem \cite{tao2d}.
In general dimension, the multilinear reduction follows from a Bourgain-Guth argument (see \cite{bg} and \cite[Section 8]{sg}).
We note that if $a$ and $b$ are so large that $\nu^{a}, \nu^{b} \approx \delta$ then $M_{a, b} \approx O(1)$ and so the goal of the iteration
is to efficiently move from small $a$ and $b$ to very large $a$ and $b$.

Lemma \ref{freq_outline} is the most technical of the four lemmas and is where we use a Fefferman-Cordoba argument in Section \ref{fc}.
We can still close the iteration with Lemma \ref{freq_outline} replaced by $M_{a, b} \lsm M_{b, b}$
for $1 \leq a < b$ and $M_{b, b} \lsm \nu^{-1/6}M_{2b, b}$. Both these estimates come from the same
proof of Lemma \ref{freq_outline} and is how we approach the iteration in Sections \ref{unc} and \ref{bik}
(see Lemmas \ref{mglem2} and \ref{mglem1} and their rigorous counterparts Lemmas \ref{mglem2_rig} and \ref{mglem1_rig}).
The proof of these lemmas is a consequence of $l^2 L^2$ decoupling and ball inflation.
Finally, Lemma \ref{switch_outline} is an application of H\"{o}lder's inequality and parabolic rescaling.

\subsection{Comparison with efficient congruencing as in \cite[Section 4]{pierce}}
The main object of iteration in \cite[Section 4]{pierce} is the following bilinear object
\begin{align*}
&I_{1}(X; a, b)\\
& = \max_{\xi \neq \eta\Mod{p}}\int_{(0, 1]^k}|\sum_{\st{1 \leq x \leq X\\x \equiv \xi \Mod{p^a}}}e(\alpha_1 x + \alpha_2 x^2)|^{2}|\sum_{\st{1 \leq y \leq X\\y \equiv \eta \Mod{p^a}}}e(\alpha_1 y + \alpha_2 y^2)|^{4}\, d\alpha.
\end{align*}
Lemma \ref{parab_outline}-\ref{switch_outline} correspond directly to Lemmas 4.2-4.5 of \cite[Section 4]{pierce}.
The observation that Lemmas 4.2 and 4.3 of \cite{pierce} correspond to parabolic rescaling and bilinear
reduction, respectively was already observed by Pierce in \cite[Section 8]{pierce}.

We think of $p$ as $\nu^{-1}$, $J(X)/X^{3}$ as $D(\delta)$, and $p^{a + 2b}I_{1}(X; a, b)/X^{3}$ as $M_{a, b}(\delta, \nu)^{6}$.
We have the expressions $J(X)/X^3$ and $p^{a + 2b}I_{1}(X; a, b)/X^3$ because heuristically assuming square root cancellation (ignoring $X^{\vep}$ powers) we expect
$J(X) \approx X^3$ and $I_{1}(X; a, b) \approx X^{3}/p^{a + 2b}$.
This heuristic explains why $$I_{1}(X; a, b) \leq p^{2b - a}I_{1}(X; 2b, b)$$ from \cite[Lemma 4.4]{pierce}
becomes (essentially, after ignoring the $\nu^{-1} \approx \delta^{-\vep}$) $$M_{a, b}(\delta, \nu)^{6} \lsm M_{2b, b}(\delta, \nu)^{6}.$$

In the definition of $I_1$, the $\max_{\xi \neq \eta\Mod{p}}$ condition
can be thought of as corresponding to the transversality condition
that $I_1$ and $I_2$ are $\nu$-separated intervals
of length $\nu$. The integral over $(0, 1]^2$ corresponds to an integral over $B$.
Finally the expression $$|\sum_{\st{1 \leq x \leq X\\x \equiv \xi \Mod{p^a}}}e(\alpha_1 x + \alpha_2 x^2)|,$$
can be thought of as corresponding to $|\E_{I}g|$ for
$I$ an interval of length $\nu^a$ and so the whole of $I_{1}(X; a, b)$ can be thought of as
$\int_{B}|\E_{I_1}g|^{2}|\E_{I_2}g|^{4}$ where $\ell(I_1) = \nu^a$ and $\ell(I_2) = \nu^b$ with $I_1$ and $I_2$ are $O(\nu)$-separated.
This will be our interpretation in Section \ref{fc}.

Interpreting the proof of Lemma \ref{freq_outline} using the uncertainty principle, we
reinterpret $I_{1}(X; a, b)$ as (ignoring weight functions)
\begin{align}\label{iab_interpret}
\avg{\Delta \in P_{\nu^{-\max(a, b)}}(B)}\nms{\E_{I}g}_{L^{2}_{\#}(\Delta)}^{2}\nms{\E_{I'}g}_{L^{4}_{\#}(\Delta)}^{4}
\end{align}
where $I$ and $I'$ are length $\nu^{a}$ and $\nu^{b}$, respectively and are $\nu$-separated.
The uncertainty principle says that \eqref{iab_interpret} is essentially equal to
$\frac{1}{|B|}\int_{B}|\E_{I}g|^{2}|\E_{I'}g|^{4}$.

Finally in Section \ref{bds} we replace \eqref{iab_interpret} with
\begin{align*}
\avg{\Delta \in P_{\nu^{-b}}(B)}(\sum_{J \in P_{\nu^{b}}(I)}\nms{\E_{J}g}_{L^{2}_{\#}(\Delta)}^{2})(\sum_{J' \in P_{\nu^{b}}(I')}\nms{\E_{J'}g}_{L^{2}_{\#}(\Delta)}^{2})^{2}
\end{align*}
where $I$ and $I'$ are length $\nu$ and $\nu$-separated. Note that when $b = 1$ this then is exactly equal to
$\frac{1}{|B|}\int_{B}|\E_{I}g|^{2}|\E_{I'}g|^{4}$. The interpretation given above is now
similar to the $A_p$ object studied by Bourgain-Demeter in \cite{sg}.

\subsection{Overview}
Theorem \ref{ef2d_main} will be proved in Section \ref{fc} via a Fefferman-Cordoba argument. This argument
does not generalize to proving that $D_{p}(\delta) \lsm_{\vep} \delta^{-\vep}$ except for $p = 4, 6$.
However in Section \ref{unc}, by the uncertainty principle we reinterpret a key lemma from Section 2 (Lemma \ref{abup}) which
allows us to generalize the argument in Section \ref{fc} so that it can work for all $2 \leq p \leq 6$.
We make this completely rigorous in Section \ref{bik} by defining a slightly different (but morally equivalent)
bilinear decoupling constant. A basic version of the ball inflation inequality similar to that used in \cite[Theorem 9.2]{sg}
and \cite[Theorem 6.6]{bdg} makes an appearance.
Finally in Section \ref{bds}, we reinterpret the argument made in Section \ref{bik} and write an argument
that is more like that given in \cite{sg}. We create a 1-parameter family of bilinear constants which
in some sense ``interpolate" between the Bourgain-Demeter argument and our argument here.

The three arguments in Sections \ref{fc}-\ref{bds} are similar but will use slightly different
bilinear decoupling constants.
We will only mention explicit constants in Section \ref{fc}.
In Sections \ref{bik} and \ref{bds}, for simplicity, we will only prove that $D(\delta) \lsm_{\vep} \delta^{-\vep}$.
Because the structure of the iteration in Sections \ref{bik} and \ref{bds}
is the same as that in Section \ref{fc}, we obtain essentially the same quantitative bounds as in Theorem \ref{ef2d_main} when making explicit
the bounds in Sections \ref{bik} and \ref{bds}.

Finally, in Section \ref{explicitdis}, we include some discussion on the explicit constants for various estimates that we need for the proof of Theorem \ref{ef2d_main}.

\subsubsection*{Acknowledgements}
The author would like to thank Ciprian Demeter, Larry Guth, and his advisor Terence Tao for encouragement and many discussions
on decoupling. The author would also like to thank Kevin Hughes and Trevor Wooley for a fruitful discussion
on efficient congruencing at the \emph{Harmonic Analysis and Related Areas} conference held by the Clay Math Institute at the University of Oxford in
September 2017. The author is partially supported by NSF grants DGE-1144087 and DMS-1266164.
Finally the author would like to thank the referee for detailed comments and suggestions.

\section{Proof of Theorem \ref{ef2d_main}}\label{fc}
We recall the definition of the bilinear decoupling constant $M_{a, b}$ as in \eqref{mabdef}.
The arguments in this section will rely strongly on the fact that the exponents in the definition of $M_{a, b}$ are 2 and 4, though we
will only essentially use this in Lemma \ref{abup}.

Given two expressions $x_1$ and $x_2$, let
$$\geom_{2, 4} x_i := x_{1}^{2/6}x_{2}^{4/6}.$$
H\"{o}lder's inequality gives $\nms{\geom_{2, 4} x_i}_{p} \leq \geom_{2, 4}\nms{x_i}_{p}$.

\subsection{Parabolic rescaling and consequences}
The linear decoupling constant $D(\delta)$ obeys the following important property.
\begin{lemma}[Parabolic rescaling]\label{parab_res}
Let $0 < \delta < \sigma < 1$ be such that $\sigma, \delta, \delta/\sigma \in \N^{-1}$. Let $I$ be an arbitrary interval in $[0, 1]$ of length $\sigma$. Then
\begin{align*}
\nms{\E_{I}g}_{L^{6}(B)} \leq 10^{16000} D(\frac{\delta}{\sigma})(\sum_{J \in P_{\delta}(I)}\nms{\E_{J}g}_{L^{6}(w_B)}^{2})^{1/2}
\end{align*}
for every $g: [0, 1] \rightarrow \C$ and every square $B$ of side length $\delta^{-2}$.
\end{lemma}
\begin{proof}
The proof without explicit constants is standard, see \cite[Proposition 7.1]{sg} or \cite[Lemma 3.2]{2ddec}.
The derivation of the constant $10^{16000}$ is given later in Section \ref{parab_res_const_sec} (and a similar proof can also be found in \cite[Section 2.4]{thesis},
see also that section for a minor clarification of parabolic rescaling with weight $w_B$ and the argument in \cite[Proposition 7.1]{sg}).
\end{proof}

As an immediate application of parabolic rescaling we have almost multiplicativity of the decoupling constant.
\begin{lemma}[Almost multiplicativity]\label{almostmult}
Let $0 < \delta < \sigma < 1$ be such that $\sigma, \delta, \delta/\sigma \in \N^{-1}$, then
$$D(\delta) \leq 10^{20000}D(\sigma)D(\delta/\sigma).$$
\end{lemma}
\begin{proof}
From the definition of $D(\sigma)$, we have
\begin{align*}
\nms{\E_{[0, 1]}g}_{L^{6}(B)} \leq D(\sigma)(\sum_{J \in P_{\sigma}([0, 1])}\nms{\E_{J}g}_{L^{6}(w_B)}^{2})^{1/2}.
\end{align*}
For each $J \in P_{\sigma}([0, 1])$, combining Lemma \ref{parab_res} with Corollary \ref{up} gives
\begin{align*}
\nms{\E_{J}g}_{L^{6}(w_B)} \leq 8^{100/6}10^{16000}D(\frac{\delta}{\sigma})(\sum_{J' \in P_{\delta}(J)}\nms{\E_{J'}g}_{L^{6}(w_B)}^{2})^{1/2}.
\end{align*}
\end{proof}

The trivial bound of $O(\nu^{(a + 2b)/6}\delta^{-1/2})$ for $M_{a, b}(\delta, \nu)$
is too weak for applications. We instead give another trivial bound that
follows from parabolic rescaling.
\begin{lemma}\label{mabtriv}
If $\delta$ and $\nu$ were such that $\nu^{a}\delta^{-1}, \nu^{b}\delta^{-1} \in \N$, then
$$M_{a, b}(\delta, \nu) \leq 10^{20000}D(\frac{\delta}{\nu^a})^{1/3}D(\frac{\delta}{\nu^b})^{2/3}.$$
\end{lemma}
\begin{proof}
Fix arbitrary $I_1 \in P_{\nu^{a}}([0, 1])$ and $I_2 \in P_{\nu^{b}}([0, 1])$
which are $3\nu$-separated. H\"{o}lder's inequality gives that
\begin{align*}
\nms{\geom_{2, 4}|\E_{I_i}g|}_{L^{6}(B)}^{6} \leq \nms{\E_{I_1}g}_{L^{6}(B)}^{2}\nms{\E_{I_2}g}_{L^{6}(B)}^{4}.
\end{align*}
Parabolic rescaling bounds this by
\begin{align*}
10^{120000}D(\frac{\delta}{\nu^a})^{2}D(\frac{\delta}{\nu^b})^{4}(\sum_{J \in P_{\delta}(I_1)}\nms{\E_{J}g}_{L^{6}(w_B)}^{2})(\sum_{J' \in P_{\delta}(I_2)}\nms{\E_{J'}g}_{L^{6}(w_B)}^{2})^{2}.
\end{align*}
Taking sixth roots then completes the proof of Lemma \ref{mabtriv}.
\end{proof}

H\"{o}lder's inequality and parabolic rescaling allows us to interchange the $a$ and $b$ in $M_{a, b}$.
\begin{lemma}\label{switch}
Suppose $b \geq 1$ and $\delta$ and $\nu$ were such that $\nu^{2b}\delta^{-1} \in \N$.
Then
\begin{align*}
M_{2b, b}(\delta, \nu) \leq 10^{10000}M_{b, 2b}(\delta, \nu)^{1/2}D(\delta/\nu^b)^{1/2}.
\end{align*}
\end{lemma}
\begin{proof}
Fix arbitrary $I_1$ and $I_2$ intervals of lengths $\nu^{2b}$ and $\nu^{b}$, respectively which
are $3\nu$-separated. H\"{o}lder's inequality then gives
\begin{align*}
\nms{|\E_{I_1}g|^{1/3}|\E_{I_2}g|^{2/3}}_{L^{6}(B)}^{6} \leq (\int_{B}|\E_{I_1}g|^{4}|\E_{I_2}g|^{2})^{1/2}(\int_{B}|\E_{I_2}g|^{6})^{1/2}.
\end{align*}
Applying the definition of $M_{b, 2b}$ and parabolic rescaling bounds the above by
\begin{align*}
(10^{20000})^{3}M_{b, 2b}(\delta, \nu)^{3}D(\frac{\delta}{\nu^b})^{3}(\sum_{J \in P_{\delta}(I_1)}\nms{\E_{J}g}_{L^{6}(w_B)}^{2})(\sum_{J' \in P_{\delta}(I_2)}\nms{\E_{J'}g}_{L^{6}(w_B)}^{2})^{2}
\end{align*}
which completes the proof of Lemma \ref{switch}.
\end{proof}

\begin{lemma}[Bilinear reduction]\label{biv1}
Suppose $\delta$ and $\nu$ were such that $\nu\delta^{-1} \in \N$. Then
$$D(\delta) \leq 10^{30000}(D(\frac{\delta}{\nu}) + \nu^{-1}M_{1, 1}(\delta, \nu)).$$
\end{lemma}
\begin{proof}
Let $\{I_i\}_{i = 1}^{\nu^{-1}} = P_{\nu}([0, 1])$.
We have
\begin{align}\label{ef2d_bieq1}
\nms{\E_{[0, 1]} g}_{L^{6}(B)} &= \nms{\sum_{1 \leq i \leq \nu^{-1}}\E_{I_i}g}_{L^{6}(B)} \leq \nms{\sum_{1 \leq i, j \leq \nu^{-1}}|\E_{I_i}g||\E_{I_j}g|}_{L^{3}(B)}^{1/2}\nonumber\\
&\leq \sqrt{2}\bigg( \nms{\sum_{\st{1 \leq i, j \leq \nu^{-1}\\|i - j| \leq 3}}|\E_{I_i}g||\E_{I_j}g|}_{L^{3}(B)}^{1/2} + \nms{\sum_{\st{1 \leq i, j \leq \nu^{-1}\\|i - j| > 3}}|\E_{I_i}g||\E_{I_j}g|}_{L^{3}(B)}^{1/2}\bigg).
\end{align}
We first consider the diagonal terms. The triangle inequality followed by
the Cauchy-Schwarz inequality gives that
\begin{align*}
\nms{\sum_{\st{1 \leq i, j \leq \nu^{-1}\\|i - j| \leq 3}}|\E_{I_i}g||\E_{I_j}g|}_{L^{3}(B)} \leq \sum_{\st{1 \leq i, j \leq \nu^{-1}\\|i - j| \leq 3}}\nms{\E_{I_i}g}_{L^{6}(B)}\nms{\E_{I_j}g}_{L^{6}(B)}.
\end{align*}
Parabolic rescaling and the Cauchy-Schwarz inequality bounds this by
\begin{align*}
10^{40000}&D(\frac{\delta}{\nu})^{2}\sum_{\st{1 \leq i, j \leq \nu^{-1}\\|i - j| \leq 3}}(\sum_{J \in P_{\delta}(I_i)}\nms{\E_{J}g}_{L^{6}(w_B)}^{2})^{1/2}(\sum_{J \in P_{\delta}(I_j)}\nms{\E_{J}g}_{L^{6}(w_B)}^{2})^{1/2}\\
&\leq 10^{40010} D(\frac{\delta}{\nu})^{2}\sum_{J \in P_{\delta}([0, 1])}\nms{\E_{J}g}_{L^{6}(w_B)}^{2}.
\end{align*}
Therefore the first term in \eqref{ef2d_bieq1} is bounded above by
\begin{align}\label{diagbd}
10^{30000}D(\frac{\delta}{\nu})(\sum_{J \in P_{\delta}([0, 1])}\nms{\E_{J}g}_{L^{6}(w_B)}^{2})^{1/2}.
\end{align}

Next we consider the off-diagonal terms. We have
\begin{align*}
\nms{\sum_{\st{1 \leq i, j \leq \nu^{-1}\\|i - j| > 3}}|\E_{I_i}g||\E_{I_j}g|}_{L^{3}(B)}^{1/2} \leq \nu^{-1}\max_{\st{1 \leq i, j \leq \nu^{-1}\\|i - j| > 3}}\nms{|\E_{I_i}g||\E_{I_j}g|}_{L^{3}(B)}^{1/2}
\end{align*}
H\"{o}lder's inequality gives that
\begin{align}\label{holder_bi}
\nms{|\E_{I_i}g||\E_{I_j}g|}_{L^{3}(B)}^{1/2} \leq \nms{|\E_{I_i}g|^{1/3}|\E_{I_j}g|^{2/3}}_{L^{6}(B)}^{1/2}\nms{|\E_{I_i}g|^{2/3}|\E_{I_j}g|^{1/3}}_{L^{6}(B)}^{1/2}
\end{align}
and therefore from \eqref{mabdef} (and using that $\nu\delta^{-1} \in \N$), the second term in \eqref{ef2d_bieq1} is bounded by
$$\sqrt{2}\nu^{-1}M_{1, 1}(\delta, \nu)(\sum_{J \in P_{\delta}([0, 1])}\nms{\E_{J}g}_{L^{6}(w_B)}^{2})^{1/2}.$$
Combining this with \eqref{diagbd} and applying the definition of $D(\delta)$
then completes the proof of Lemma \ref{biv1}.
\end{proof}

\subsection{A Fefferman-Cordoba argument}
In the proof of Lemma \ref{abup} we need a version of $M_{a, b}$
with both sides being $L^{6}(w_B)$. The following lemma shows that these two constants are equivalent.
\begin{lemma}\label{wuw}
Suppose $\delta$ and $\nu$ were such that $\nu^{a}\delta^{-1}$, $\nu^{b}\delta^{-1} \in \N$.
Let $M_{a, b}'(\delta, \nu)$ be the best constant such that
\begin{align*}
\int |\E_{I}g|^{2}|\E_{I'}g|^{4}w_{B} \leq M_{a, b}'(\delta, \nu)^{6}(\sum_{J \in P_{\delta}(I)}\nms{\E_{J}g}_{L^{6}(w_B)}^{2})(\sum_{J' \in P_{\delta}(I')}\nms{\E_{J'}g}_{L^{6}(w_B)}^{2})^{2}
\end{align*}
for all squares $B$ of side length $\delta^{-2}$, $g: [0, 1] \rightarrow \C$, and all $3\nu$-separated intervals
$I \in P_{\nu^{a}}([0, 1])$ and $I' \in P_{\nu^{b}}([0, 1])$. Then
\begin{align*}
M_{a, b}'(\delta, \nu) \leq 12^{100/6}M_{a, b}(\delta, \nu).
\end{align*}
\end{lemma}
\begin{rem}
Since $1_B \lsm w_B$, we find $M_{a, b}(\delta, \nu) \lsm M_{a, b}'(\delta, \nu)$ and hence Lemma \ref{wuw} implies
$M_{a, b} \sim M_{a, b}'$.
\end{rem}
\begin{proof}
Fix arbitrary $3\nu$-separated intervals $I_1 \in P_{\nu^{a}}([0, 1])$ and $I_2 \in P_{\nu^{b}}([0, 1])$.
It suffices to assume that $B$ is centered at the origin.

Corollary \ref{convint} gives
\begin{align*}
\nms{\geom_{2, 4}|\E_{I_i}g|}_{L^{6}(w_{B})}^{6} \leq 3^{100}\int_{\R^2}\nms{\geom_{2, 4}|\E_{I_i}g|}_{L^{6}_{\#}(B(y, \delta^{-2}))}^{6}w_{B}(y)\, dy.
\end{align*}
Unraveling the definition of $L^{6}_{\#}$ and applying the definition of $M_{a, b}$ gives that the above is
\begin{align*}
&\leq 3^{100}\delta^{4}M_{a, b}(\delta, \nu)^{6}\int_{\R^2}\geom_{2, 4}(\sum_{J \in P_{\delta}(I_i)}\nms{\E_{J}g}_{L^{6}(w_{B(y, \delta^{-2})})}^{2})^{3}w_{B}(y)\, dy\\
&\leq 3^{100}\delta^{4}M_{a, b}(\delta, \nu)^{6}\geom_{2, 4}\int_{\R^2}(\sum_{J \in P_{\delta}(I_i)}\nms{\E_{J}g}_{L^{6}(w_{B(y, \delta^{-2})})}^{2})^{\frac{1}{2}\cdot 6}w_{B}(y)\, dy\\
&\leq 3^{100}\delta^{4}M_{a, b}(\delta, \nu)^{6}\geom_{2, 4}(\sum_{J \in P_{\delta}(I_i)}(\int_{\R^2}\nms{\E_{J}g}_{L^{6}(w_{B(y, \delta^{-2})})}^{6}w_{B}(y)\, dy)^{1/3})^{3}
\end{align*}
where the second inequality is by H\"{o}lder's inequality and the third inequality is by Minkowski.
Since $B$ is centered at the origin, $w_B \ast w_B \leq 4^{100}\delta^{-4}w_B$ (Lemma \ref{wbconvolve}) and hence
\begin{align*}
\delta^{4}\int_{\R^2}\nms{\E_{J}g}_{L^{6}(w_{B(y, \delta^{-2})})}^{6}w_{B}(y)\, dy \leq 4^{100} \nms{\E_{J}g}_{L^{6}(w_{B})}^{6}.
\end{align*}
This then immediately implies that $M_{a, b}'(\delta, \nu) \leq 12^{100/6} M_{a, b}(\delta, \nu)$ which
completes the proof of Lemma \ref{wuw}.
\end{proof}

We have the following key technical lemma of this paper.
We encourage the reader to compare the argument with that of \cite[Lemma 4.4]{pierce}.

\begin{lemma}\label{abup}
Let $a$ and $b$ be integers such that $1 \leq a \leq 2b$.
Suppose $\delta$ and $\nu$ was such that $\nu^{2b}\delta^{-1} \in \N$.
Then
$$M_{a, b}(\delta, \nu) \leq 10^{1000} \nu^{-1/6}M_{2b, b}(\delta, \nu).$$
\end{lemma}
\begin{proof}
It suffices to assume that $B$ is centered at the origin with side length $\delta^{-2}$.
The integrality conditions on $\delta$ and $\nu$ imply that $\delta \leq \nu^{2b}$
and $\nu^{a}\delta^{-1}, \nu^{b}\delta^{-1} \in \N$.
Fix arbitrary intervals $I_1 = [\alpha, \alpha + \nu^a] \in P_{\nu^{a}}([0, 1])$ and
$I_2 = [\beta, \beta + \nu^b] \in P_{\nu^{b}}([0, 1])$ which are $3\nu$-separated.

Let $g_{\beta}(x) := g(x + \beta)$, $T_{\beta} = (\begin{smallmatrix} 1 & 2\beta\\0 & 1\end{smallmatrix})$, and $d := \alpha - \beta$. Shifting
$I_2$ to $[0, \nu^b]$ gives that
\begin{align}\label{cov}
\int_{B}|(\E_{I_1}g)(x)|^{2}|(\E_{I_2}g)(x)|^{4}\, dx &= \int_{B}|(\E_{[d, d + \nu^a]}g_{\beta})(T_{\beta}x)|^{2}|(\E_{[0, \nu^b]}g_{\beta})(T_{\beta}x)|^{4}\, dx\nonumber\\
& = \int_{T_{\beta}(B)}|(\E_{[d, d + \nu^a]}g_{\beta})(x)|^{2}|(\E_{[0, \nu^b]}g_{\beta})(x)|^{4}\, dx.
\end{align}
Note that $d$ can be negative; however since $g: [0, 1] \rightarrow \C$ and $d = \alpha - \beta$, $\E_{[d, d + \nu^a]}g_{\beta}$ is defined.
Since $|\beta| \leq 1$, $T_{\beta}(B) \subset 100B$. Combining this with $1_{100B} \leq \eta_{100B}$ gives that \eqref{cov} is
\begin{align}\label{target0}
&\leq \int_{\R^2} |(\E_{[d, d + \nu^a]}g_{\beta})(x)|^{2}|(\E_{[0, \nu^b]}g_{\beta})(x)|^{4}\eta_{100B}(x)\, dx\nonumber\\
& = \sum_{J_1, J_2 \in P_{\nu^{2b}}([d, d + \nu^a])}\int_{\R^2}(\E_{J_1}g_{\beta})(x)\ov{(\E_{J_2}g_{\beta})(x)}|(\E_{[0, \nu^b]}g_{\beta})(x)|^{4}\eta_{100B}(x)\, dx.
\end{align}
We claim that if $d(J_1, J_2) > 10\nu^{2b - 1}$, the integral in \eqref{target0} is equal to 0.

Suppose $J_1, J_2 \in P_{\nu^{2b}}([d, d + \nu^a])$ such that $d(J_1, J_2) > 10\nu^{2b - 1}$.
Expanding the integral in \eqref{target0} for this pair of $J_1, J_2$ gives that it is equal to
\begin{equation}\label{target1}
\int_{\R^2}\bigg(\int_{J_1 \times [0, \nu^b]^2 \times J_2 \times [0, \nu^b]^2}\prod_{i = 1}^{3}g_{\beta}(\xi_i)\ov{g_{\beta}(\xi_{i + 3})}e(\cdots)\, \prod_{i = 1}^{6}d\xi_i\bigg)\eta_{100B}(x)\, dx
\end{equation}
where the expression inside the $e(\cdots)$ is
$$((\xi_1 - \xi_4)x_1 + (\xi_{1}^{2} - \xi_{4}^{2})x_2) + ((\xi_2 + \xi_3 - \xi_5 - \xi_6)x_1 + (\xi_{2}^{2} + \xi_{3}^{2} -\xi_{5}^{2} - \xi_{6}^{2})x_2).$$
Interchanging the integrals in $\xi$ and $x$ shows that the integral in $x$ is equal to the inverse Fourier transform of $\eta_{100B}$ evaluated at
\begin{align*}
(\sum_{i = 1}^{3}(\xi_{i} - \xi_{i + 3}), \sum_{i = 1}^{3}(\xi_{i}^{2} - \xi_{i + 3}^{2})).
\end{align*}
Since the inverse Fourier transform of $\eta_{100B}$ is supported in $B(0, \delta^{2}/100)$, \eqref{target1} is equal to 0
unless
\begin{align}\label{f2}
|\sum_{i = 1}^{3}(\xi_{i} - \xi_{i + 3})| &\leq \delta^{2}/200\nonumber\\
|\sum_{i = 1}^{3}(\xi_{i}^{2} - \xi_{i + 3}^{2})| & \leq \delta^{2}/200.
\end{align}
Since $\delta \leq \nu^{2b}$ and $\xi_{i} \in [0, \nu^b]$ for $i = 2, 3, 5, 6$, \eqref{f2} implies
\begin{align}\label{f3}
|\xi_1 - \xi_4||\xi_1 + \xi_4| = |\xi_{1}^{2} - \xi_{4}^{2}| \leq 5\nu^{2b}.
\end{align}
Since $I_1, I_2$ are $3\nu$-separated, $|d| \geq 3\nu$. Recall that $\xi_1 \in J_1$, $\xi_4 \in J_2$ and $J_1, J_2$
are subsets of $[d, d + \nu^a]$. Write $\xi_1 = d + r$ and $\xi_4 = d + s$ with $r, s \in [0, \nu^a]$.
Then
\begin{align}\label{f4}
|\xi_1 + \xi_4| = |2d + (r + s)| \geq 6\nu - |r + s| \geq 6\nu - 2\nu^{a} \geq 4\nu.
\end{align}
Since $d(J_1, J_2) > 10\nu^{2b - 1}$, $|\xi_1 - \xi_4| > 10\nu^{2b - 1}$.
Therefore the left hand side of \eqref{f3} is $> 40\nu^{2b}$, a contradiction. Thus
the integral in \eqref{target0} is equal to 0 when $d(J_1, J_2) > 10\nu^{2b - 1}$.

The above analysis implies that \eqref{target0} is
\begin{align*}
\leq \sum_{\st{J_1, J_2 \in P_{\nu^{2b}}([d, d + \nu^a])\\d(J_1, J_2) \leq 10\nu^{2b - 1}}}\int_{\R^2}|(\E_{J_1}g_{\beta})(x)||(\E_{J_2}g_{\beta})(x)||(\E_{[0, \nu^b]}g_{\beta})(x)|^{4}\eta_{100B}(x)\, dx.
\end{align*}
Undoing the change of variables as in \eqref{cov} gives that the above is equal to
\begin{align}\label{pent1}
\sum_{\st{J_1, J_2 \in P_{\nu^{2b}}(I_1)\\d(J_1, J_2) \leq 10\nu^{2b - 1}}}\int_{\R^2}|(\E_{J_1}g)(x)||(\E_{J_2}g)(x)||(\E_{I_2}g)(x)|^{4}\eta_{100B}(T_{\beta}x)\, dx.
\end{align}
Observe that
\begin{align*}
\eta_{100B}(T_{\beta}x) \leq 10^{2400}w_{100B}(T_{\beta}x) \leq 10^{2600} w_{100B}(x) \leq 10^{2800} w_{B}(x)
\end{align*}
where the second inequality is an application of Lemma \ref{shear} and the last inequality is because $w_{B}(x)^{-1}w_{100B}(x) \leq 10^{200}$.
An application of the Cauchy-Schwarz inequality shows that \eqref{pent1} is
\begin{align*}
\leq 10^{2800}\sum_{\st{J_1, J_2 \in P_{\nu^{2b}}(I_1)\\d(J_1, J_2) \leq 10\nu^{2b - 1}}}(\int_{\R^2}|\E_{J_1}g|^{2}|\E_{I_2}g|^{4}w_B)^{1/2}(\int_{\R^2}|\E_{J_2}g|^{2}|\E_{I_2}g|^{4}w_{B})^{1/2}.
\end{align*}
Note that for each $J_1\in P_{\nu^{2b}}(I_1)$, there are $\leq 10000\nu^{-1}$ intervals $J_2 \in P_{\nu^{2b}}(I_1)$
such that $d(J_1, J_2) \leq 10\nu^{2b - 1}$.
Thus two applications of the Cauchy-Schwarz inequality bounds the above by
\begin{align*}
10^{2802}\nu^{-1/2}&(\sum_{J_1 \in P_{\nu^{2b}}(I_1)}\int_{\R^2}|\E_{J_1}g|^{2}|\E_{I_2}g|^{4}w_B)^{1/2}\times\\
&\hspace{1in}(\sum_{J_1 \in P_{\nu^{2b}}(I_1)}\sum_{\st{J_2 \in P_{\nu^{2b}}(I_2)\\d(J_1, J_2) \leq 10\nu^{2b - 1}}}\int_{\R^2}|\E_{J_2}g|^{2}|\E_{I_2}g|^{4}w_B)^{1/2}.
\end{align*}
Since there are $\leq 10000\nu^{-1}$ relevant $J_2$ for each $J_1$, the above is
\begin{align*}
&\leq 10^{3000} \nu^{-1}\sum_{J \in P_{\nu^{2b}}(I_1)}\int_{\R^2}|\E_{J}g|^{2}|\E_{I_2}g|^{4}w_B\\
&\leq 10^{3000}12^{100}\nu^{-1}M_{2b, b}(\delta, \nu)^{6}(\sum_{J \in P_{\delta}(I_1)}\nms{\E_{J}g}_{L^{6}(w_B)}^{2})(\sum_{J' \in P_{\delta}(I_2)}\nms{\E_{J'}g}_{L^{6}(w_B)}^{2})^{2}
\end{align*}
where the last inequality is an application of Lemma \ref{wuw}.
This completes the proof of Lemma \ref{abup}.
\end{proof}

Iterating Lemmas \ref{switch} and \ref{abup} repeatedly gives the following estimate.
\begin{lemma}\label{m11iter}
Let $N \in \N$ and suppose $\delta$ and $\nu$ were such that $\nu^{2^{N}}\delta^{-1} \in \N$.
Then
\begin{align*}
M_{1, 1}(\delta, \nu) \leq 10^{60000}\nu^{-1/3}D(\frac{\delta}{\nu^{2^{N-1}}})^{\frac{1}{3\cdot 2^{N}}}D(\frac{\delta}{\nu^{2^{N}}})^{\frac{2}{3\cdot 2^{N}}}\prod_{j = 0}^{N-1}D(\frac{\delta}{\nu^{2^j}})^{1/2^{j + 1}}.
\end{align*}
\end{lemma}
\begin{proof}
Lemmas \ref{switch} and \ref{abup} imply that if $1 \leq a \leq 2b$ and $\delta$ and $\nu$ were such that $\nu^{2b}\delta^{-1} \in \N$,
then
\begin{align}\label{brev}
M_{a, b}(\delta, \nu) \leq 10^{20000}\nu^{-1/6}M_{b, 2b}(\delta, \nu)^{1/2}D(\frac{\delta}{\nu^b})^{1/2}.
\end{align}

Since $\nu^{2^{N}}\delta^{-1} \in \N$, $\nu^{i}\delta^{-1} \in \N$ for $i = 0, 1, 2, \ldots, 2^{N}$.
Applying \eqref{brev} repeatedly gives
\begin{align*}
M_{1, 1}(\delta, \nu) \leq 10^{40000}\nu^{-1/3}M_{2^{N-1}, 2^{N}}(\delta, \nu)^{\frac{1}{2^N}}\prod_{j = 0}^{N-1}D(\frac{\delta}{\nu^{2^j}})^{1/2^{j + 1}}.
\end{align*}
Bounding $M_{2^{N-1}, 2^{N}}$ using Lemma \ref{mabtriv} then completes the proof of Lemma \ref{m11iter}.
\end{proof}

\begin{rem}
A similar analysis as in \eqref{f2}-\eqref{f4} shows that if $1 \leq a < b$ and $\delta$ and $\nu$ were such that $\nu^{b}\delta^{-1} \in \N$,
then $M_{a, b}(\delta, \nu) \lsm M_{b, b}(\delta, \nu)$. Though we do not iterate this way in this section, it is enough to close
the iteration with $M_{a, b} \lsm M_{b, b}$ for $1 \leq a < b$, and $M_{b, b} \lsm \nu^{-1/6}M_{2b, b}$, and Lemma \ref{switch}.
We interpret the iteration and in particular Lemma \ref{abup} this way in Sections \ref{unc}-\ref{bds}.
\end{rem}

\subsection{The $O_{\vep}(\delta^{-\vep})$ bound}\label{fc_iter}
Combining Lemma \ref{m11iter} with Lemma \ref{biv1} gives the following.
\begin{cor}\label{decrec}
Let $N \in \N$ and suppose $\delta$ and $\nu$ were such that $\nu^{2^{N}}\delta^{-1} \in \N$.
Then
\begin{align*}
D(\delta) \leq 10^{10^{5}}\bigg(D(\frac{\delta}{\nu}) + \nu^{-4/3}D(\frac{\delta}{\nu^{2^{N-1}}})^{\frac{1}{3\cdot 2^{N}}}D(\frac{\delta}{\nu^{2^{N}}})^{\frac{2}{3\cdot 2^{N}}}\prod_{j = 0}^{N-1}D(\frac{\delta}{\nu^{2^j}})^{1/2^{j + 1}}\bigg)
\end{align*}
\end{cor}

Choosing $\nu = \delta^{1/2^{N}}$ in Corollary \ref{decrec}
and requiring that $\nu=\delta^{1/2^{N}} \in \N^{-1} \cap (0, 1/100)$
gives the following result.
\begin{cor}\label{core}
Let $N \in \N$ and suppose $\delta$ was such that $\delta^{-1/2^{N}} \in \N$ and
$\delta < 100^{-2^{N}}$. Then
\begin{align*}
D(\delta) \leq 10^{10^{5}}\bigg(D(\delta^{1 - \frac{1}{2^{N}}}) + \delta^{-\frac{4}{3\cdot 2^{N}}}D(\delta^{1/2})^{\frac{1}{3\cdot 2^{N}}}\prod_{j = 0}^{N-1}D(\delta^{1 - \frac{1}{2^{N - j}}})^{\frac{1}{2^{j + 1}}}\bigg).
\end{align*}
\end{cor}

Corollary \ref{core} allows us to conclude that $D(\delta) \lsm_{\vep} \delta^{-\vep}$.
To see this, the trivial bounds for $D(\delta)$ are $1 \lsm D(\delta) \lsm \delta^{-1/2}$ for all $\delta \in \N^{-1}$.
Let $\ld$ be the smallest real number such that $D(\delta) \lsm_{\vep} \delta^{-\ld - \vep}$ for all $\delta \in \N^{-1}$.
From the trivial bounds, $\ld \in [0, 1/2]$. We claim that $\ld = 0$.
Suppose $\ld > 0$.

Choose $N$ to be an integer such that
\begin{align}\label{ef2d_mchoice}
\frac{5}{6} + \frac{N}{2} - \frac{4}{3\ld}\geq 1.
\end{align}
Then by Corollary \ref{core}, for $\delta^{-1/2^{N}} \in \N$ with $\delta < 100^{-2^{N}}$,
\begin{align*}
D(\delta) &\lsm_{\vep} \delta^{-\ld(1 - \frac{1}{2^{N}}) - \vep} + \delta^{-\frac{4}{3\cdot 2^{N}} - \frac{\ld}{6\cdot 2^{N}} - \sum_{j = 0}^{N-1}(1 - \frac{1}{2^{N - j}})\frac{\ld}{2^{j + 1}}- \vep}\\
&\lsm_{\vep} \delta^{-\ld(1 - \frac{1}{2^{N}}) - \vep} + \delta^{-\ld(1 - (\frac{5}{6} + \frac{N}{2} - \frac{4}{3\ld})\frac{1}{2^{N}}) - \vep} \lsm_{\vep} \delta^{-\ld(1 - \frac{1}{2^{N}}) - \vep}
\end{align*}
where in the last inequality we have used \eqref{ef2d_mchoice}.
Applying almost multiplicativity of the linear decoupling constant (similar to \cite[Section 2.10]{thesis} or the proof of Lemma \ref{expstep2} later) then shows that for all $\delta \in \N^{-1}$,
\begin{align*}
D(\delta) \lsm_{N, \vep} \delta^{-\ld(1 - \frac{1}{2^{N}}) - \vep}.
\end{align*}
This then contradicts minimality of $\ld$. Therefore $\ld = 0$ and thus we have shown that $D(\delta) \lsm_{\vep}\delta^{-\vep}$ for all $\delta \in \N^{-1}$.

\subsection{An explicit bound}
Having shown that $D(\delta) \lsm_{\vep} \delta^{-\vep}$, we now make this dependence on $\vep$ explicit.
Fix arbitrary $0 < \vep < 1/100$. Then $D(\delta) \leq C_{\vep}\delta^{-\vep}$ for all $\delta \in \N^{-1}$.
\begin{lemma}\label{expstep1}
Fix arbitrary $0 < \vep < 1/100$ and suppose $D(\delta) \leq C_{\vep}\delta^{-\vep}$ for all $\delta\in \N^{-1}$.
Let integer $N \geq 1$ be such that $$\frac{5}{6} + \frac{N}{2} - \frac{4}{3\vep} > 0.$$
Then for $\delta$ such that $\delta^{-1/2^{N}} \in \N$ and $\delta < 100^{-2^{N}}$, we have
$$D(\delta) \leq 2\cdot 10^{10^{5}}C_{\vep}^{1 - \frac{\vep}{2^{N}}}\delta^{-\vep}.$$
\end{lemma}
\begin{proof}
Inserting $D(\delta) \leq C_{\vep}\delta^{-\vep}$ into Corollary \ref{core} gives that for all integers $N \geq 1$ and $\delta$ such that
$\delta^{-1/2^{N}} \in \N$, $\delta < 100^{-2^{N}}$, we have
\begin{align*}
D(\delta) \leq 10^{10^{5}}(C_{\vep}\delta^{\frac{\vep}{2^{N}}} + C_{\vep}^{1 - \frac{2}{3 \cdot 2^N}}\delta^{\frac{\vep}{2^{N}}(\frac{5}{6} + \frac{N}{2} - \frac{4}{3\vep})})\delta^{-\vep}.
\end{align*}
Thus by our choice of $N$,
\begin{align}\label{exp1}
D(\delta) \leq 10^{10^{5}}(C_{\vep}\delta^{\frac{\vep}{2^{N}}} + C_{\vep}^{1 - \frac{2}{3 \cdot 2^N}})\delta^{-\vep}.
\end{align}
There are two possibilities. If $\delta < C_{\vep}^{-1}$, then since $0 < \vep < 1/100$, \eqref{exp1} becomes
\begin{align}\label{exp2}
D(\delta) \leq 10^{10^{5}}(C_{\vep}^{1 - \frac{\vep}{2^{N}}} + C_{\vep}^{1 - \frac{2}{3\cdot 2^{N}}})\delta^{-\vep} \leq 2\cdot 10^{10^{5}}C_{\vep}^{1 - \frac{\vep}{2^{N}}}\delta^{-\vep}.
\end{align}
On the other hand if $\delta \geq C_{\vep}^{-1}$, the trivial bound gives
\begin{align*}
D(\delta) \leq 2^{100/6}\delta^{-1/2} \leq 2^{100/6}C_{\vep}^{1/2}
\end{align*}
which is bounded above by the right hand side of \eqref{exp2}.
This completes the proof of Lemma \ref{expstep1}.
\end{proof}

Note that Lemma \ref{expstep1} is only true for $\delta$ satisfying $\delta^{-1/2^{N}} \in \N$ and
$\delta < 100^{-2^{N}}$. We now use almost multiplicativity to upgrade the result of Lemma \ref{expstep1} to all $\delta \in \N^{-1}$.
\begin{lemma}\label{expstep2}
Fix arbitrary $0 < \vep < 1/100$ and suppose $D(\delta) \leq C_{\vep}\delta^{-\vep}$ for all $\delta \in \N^{-1}$.
Then
\begin{align*}
D(\delta) \leq 10^{10^6}2^{4\cdot 8^{1/\vep}}C_{\vep}^{1 - \frac{\vep}{8^{1/\vep}}}\delta^{-\vep}
\end{align*}
for all $\delta \in \N^{-1}$.
\end{lemma}
\begin{proof}
Choose
\begin{align}\label{mchoice3}
N := \lceil \frac{8}{3\vep} - \frac{5}{3}\rceil
\end{align}
and $\delta \in \{2^{-2^{N}n}\}_{n = 7}^{\infty} = \{\delta_{n}\}_{n = 7}^{\infty}$. Then for these $\delta$, $\delta^{-1/2^{N}} \in \N$
and $\delta < 100^{-2^{N}}$. If $\delta \in (\delta_{7}, 1] \cap \N^{-1}$, then
\begin{align*}
D(\delta) \leq 2^{100/6}\delta^{-1/2} \leq 2^{100/6}2^{2^{N-1} \cdot 7}.
\end{align*}
If $\delta \in (\delta_{n + 1}, \delta_{n}]$ for some $n \geq 7$, then almost multiplicativity and Lemma \ref{expstep1} gives that
\begin{align*}
D(\delta) &\leq 10^{20000}D(\delta_n)D(\frac{\delta}{\delta_n})\\
& \leq 10^{20000}(2\cdot 10^{10^{5}}C_{\vep}^{1 - \frac{\vep}{2^{N}}}\delta_{n}^{-\vep})(2^{100/6}(\frac{\delta_n}{\delta})^{1/2})\\
&\leq 10^{10^{6}}2^{2^{N-1}}C_{\vep}^{1 - \frac{\vep}{2^{N}}}\delta^{-\vep}
\end{align*}
where $N$ is as in \eqref{mchoice3} and the second inequality we have used the trivial bound for $D(\delta/\delta_n)$.

Combining both cases above then shows that if $N$ is chosen as in \eqref{mchoice3}, then
\begin{align*}
D(\delta) \leq 10^{10^{6}}2^{7 \cdot 2^{N-1}}C_{\vep}^{1 - \frac{\vep}{2^{N}}}\delta^{-\vep}
\end{align*}
for all $\delta \in \N^{-1}$. Since we are no longer constrained by having $N\in \N$, we can increase $N$ to be $3/\vep$
and so we have that
\begin{align*}
D(\delta) \leq 10^{10^6}2^{4\cdot 8^{1/\vep}}C_{\vep}^{1 - \frac{\vep}{8^{1/\vep}}}\delta^{-\vep}
\end{align*}
for all $\delta \in \N^{-1}$.
This completes the proof of Lemma \ref{expstep2}.
\end{proof}

\begin{lemma}\label{expstep3}
For all $0 < \vep < 1/100$ and all $\delta \in \N^{-1}$, we have
\begin{align*}
D(\delta) \leq 2^{200^{1/\vep}}\delta^{-\vep}.
\end{align*}
\end{lemma}
\begin{proof}
Let $P(C, \ld)$ be the statement that $D(\delta) \leq C\delta^{-\ld}$ for all $\delta \in \N^{-1}$.
Lemma \ref{expstep2} implies that for $\vep \in (0, 1/100)$,
\begin{align*}
P(C_{\vep}, \vep) \implies P(10^{10^6}2^{4\cdot 8^{1/\vep}}C_{\vep}^{1 - \frac{\vep}{8^{1/\vep}}}, \vep).
\end{align*}
Iterating this $M$ times gives that
\begin{align*}
P(C_{\vep}, \vep) \implies P([10^{10^6}2^{4\cdot 8^{1/\vep}}]^{\sum_{j = 0}^{M-1}(1 - \frac{\vep}{8^{1/\vep}})^{j}}C_{\vep}^{(1 - \frac{\vep}{8^{1/\vep}})^M}, \vep).
\end{align*}
Letting $M \rightarrow \infty$ thus gives that for all $0 < \vep < 1/100$,
\begin{align*}
D(\delta) \leq (10^{10^6}2^{4\cdot 8^{1/\vep}})^{8^{1/\vep}/\vep}\delta^{-\vep} \leq 2^{100^{1/\vep}/\vep}\delta^{-\vep} \leq 2^{200^{1/\vep}}\delta^{-\vep}
\end{align*}
for all $\delta \in \N^{-1}$. This completes the proof of Lemma \ref{expstep3}.
\end{proof}

Optimizing in $\vep$ then gives the proof of our main result.
\begin{proof}[Proof of Theorem \ref{ef2d_main}]
Choose $A = (\log_{2}200)(\log\frac{1}{\delta})$, $\eta = \log A - \log\log A$, and $\vep = \frac{1}{\eta}\log 200$.
Note that $\eta\exp(\eta) = A(1 - \frac{\log\log A}{\log A}) \leq A$.
Then from our choice of $\eta, A, \vep$,
$$200^{1/\vep}\log 2 \leq \vep\log\frac{1}{\delta}$$
and hence
\begin{align}\label{optimize_eq1}
2^{200^{1/\vep}}\delta^{-\vep} \leq \exp(2\vep\log\frac{1}{\delta}).
\end{align}
Since $\eta = \log A - \log\log A$, we need to ensure that our choice of $\vep$ is such that $0 < \vep < 1/100$. Thus we need
\begin{align*}
\vep = \frac{\log 200}{\log((\log_{2}200)(\log\frac{1}{\delta})) - \log\log((\log_{2}200)(\log\frac{1}{\delta}))} < \frac{1}{100}.
\end{align*}
Note that for all $x > 0$, $\log\log x < (\log x)^{1/2}$ and hence for all $0 < \delta < e^{-\frac{e^4}{\log_{2}200}}$,
\begin{align}
\log((\log_{2}200)(\log\frac{1}{\delta})) &- \log\log((\log_{2}200)(\log\frac{1}{\delta}))\nonumber\\
& \geq \log((\log_{2}200)(\log\frac{1}{\delta})) - [\log((\log_{2}200)(\log\frac{1}{\delta}))]^{1/2}\nonumber\\
& \geq \frac{1}{2}\log((\log_{2}200)(\log\frac{1}{\delta})) \geq \frac{1}{2}\log\log\frac{1}{\delta}.\label{optimize_eq2}
\end{align}
Thus we need $0 < \delta < e^{-\frac{e^4}{\log_{2}200}}$ to also be such that
\begin{align*}
\frac{2\log 200}{\log\log\frac{1}{\delta}} < \frac{1}{100}
\end{align*}
and hence $\delta < e^{-200^{200}}$. Therefore using \eqref{optimize_eq1} and \eqref{optimize_eq2}, we have that
for $\delta \in (0, e^{-200^{200}}) \cap \N^{-1}$,
\begin{align*}
D(\delta) \leq \exp(30\frac{\log\frac{1}{\delta}}{\log\log\frac{1}{\delta}}).
\end{align*}
This completes the proof of Theorem \ref{ef2d_main}.
\end{proof}

\section{An uncertainty principle interpretation of Lemma \ref{abup}}\label{unc}
We now give a different interpretation of Lemma \ref{abup}, making use of the uncertainty principle.
We will pretend all weight functions $w_B$ are indicator functions $1_B$ in this section and will
make the argument rigorous in the next section. In this section, $B$ will denote an arbitrary square of side length $\delta^{-2}$.

The main point was of Lemma \ref{abup}
was to show that if $1 \leq a \leq 2b$, $\delta$ and $\nu$ such that $\nu^{2b}\delta^{-1} \in \N$,
then
\begin{align}\label{unctar}
\int_{B}|\E_{I_1}g|^{2}|\E_{I_2}g|^{4} \lsm \nu^{-1}\sum_{J \in P_{\nu^{2b}}(I_1)}\int_{B}|\E_{J}g|^{2}|\E_{I_2}g|^{4}
\end{align}
for arbitrary $I_1 \in P_{\nu^{a}}([0, 1])$ and $I_2 \in P_{\nu^{b}}([0, 1])$
such that $d(I_1, I_2) \gtrsim \nu$.
From Lemma \ref{m11iter}, we only need \eqref{unctar} to be true for $1 \leq a \leq b$.
Our goal of this section is to prove (heuristically under the uncertainty principle) the following two statements:
\begin{enumerate}[(I)]
\item For $1 \leq a < b$, $M_{a, b}(\delta, \nu) \lsm M_{b, b}(\delta, \nu)$; in other words
\begin{align}\label{ineq1_unc}
\int_{B}|\E_{I_1}g|^{2}|\E_{I_2}g|^{4} \lsm \sum_{J \in P_{\nu^b}(I_1)}\int_{B}|\E_{J}g|^{2}|\E_{I_2}g|^{4}
\end{align}
for arbitrary $I_1 \in P_{\nu^a}([0, 1])$ and $I_2 \in P_{\nu^b}([0, 1])$ such that $d(I_1, I_2) \gtrsim \nu$.
\item $M_{b, b}(\delta, \nu) \lsm \nu^{-1/6}M_{2b, b}(\delta, \nu)$; in other words
\begin{align}\label{ineq2_unc}
\int_{B}|\E_{I_1}g|^{2}|\E_{I_2}g|^{4} \lsm \nu^{-1}\sum_{J \in P_{\nu^{2b}}(I_1)}\int_{B}|\E_{J}g|^{2}|\E_{I_2}g|^{4}
\end{align}
for arbitrary $I_1, I_2 \in P_{\nu^b}([0, 1])$ such that $d(I_1, I_2) \gtrsim \nu$.
\end{enumerate}
Replacing 4 with $p - 2$ then allows us to generalize to $2 \leq p < 6$.

The particular instance of the uncertainty principle we will use is the following.
Let $I$ be an interval of length $1/R$ with center $c$. Fix an arbitrary $R \times R^2$ rectangle $T$
oriented in the direction $(-2c, 1)$. Heuristically for $x \in T$, $(\E_{I}g)(x)$ behaves like $a_{T, I}e^{2\pi i \om_{T, I}\cdot x}$.
Here the amplitude $a_T$ depends on $g, T$, and $I$ and the phase $\om_{T, I}$ depends on $T$ and $I$.
In particular, $|(\E_{I}g)(x)|$ is
essentially constant on every $R \times R^2$ rectangle oriented in the direction
$(-2c, 1)$.
This also implies that if $\Delta$ is a square of side length $R$, then
$|(\E_{I}g)(x)|$ is essentially constant on $\Delta$ (with constant depending on $\Delta, I, g$) and
$\nms{\E_{I}g}_{L^{p}_{\#}(\Delta)}$ is essentially constant with the same constant independent of $p$.

We introduce two standard tools from \cite{sg, bdg}.
\begin{lemma}[Bernstein's inequality]\label{bern_uw}
Let $I$ be an interval of length $1/R$ and $\Delta$ a square of side length $R$.
If $1 \leq p \leq q < \infty$, then
\begin{align*}
\nms{\E_{I}g}_{L^{q}_{\#}(\Delta)} \lsm \nms{\E_{I}g}_{L^{p}_{\#}(\Delta)}.
\end{align*}
We also have $$\nms{\E_{I}g}_{L^{\infty}(\Delta)} \lsm \nms{\E_{I}g}_{L^{p}_{\#}(\Delta)}.$$
\end{lemma}
\begin{proof}
See \cite[Corollary 4.3]{sg} or \cite[Lemma 2.2.20]{thesis} for a rigorous proof.
\end{proof}
The reverse inequality in the above lemma is just an application of H\"{o}lder's inequality and so ignoring
weight functions, $\nms{\E_{I}g}_{L^{q}_{\#}(\Delta)} \sim \nms{\E_{I}g}_{L^{p}_{\#}(\Delta)}$ for any $1 \leq p, q \leq \infty$.
In other words, $\nms{\E_{I}g}_{L^{p}_{\#}(\Delta)}$ is essentially constant independent of $p$.
Therefore we can view Bernstein's inequality as one instance of the uncertainty principle.

\begin{lemma}[$l^2 L^2$ decoupling]\label{l2l2_uw}
Let $I$ be an interval of length $\geq 1/R$ such that $R|I| \in \N$ and $\Delta$ a square of side length $R$. Then
\begin{align*}
\nms{\E_{I}g}_{L^{2}(\Delta)} \lsm (\sum_{J \in P_{1/R}(I)}\nms{\E_{J}g}_{L^{2}(\Delta)}^{2})^{1/2}.
\end{align*}
\end{lemma}
\begin{proof}
See \cite[Proposition 6.1]{sg} or \cite[Lemma 2.2.21]{thesis} for a rigorous proof.
\end{proof}

The first inequality \eqref{ineq1_unc} is an immediate application of the uncertainty principle and $l^2 L^2$ decoupling.
\begin{lemma}\label{mglem2}
Suppose $1 \leq a < b$ and $\delta$ and $\nu$ were such that $\nu^{b}\delta^{-1} \in \N$.
Then
\begin{align*}
\int_{B}|\E_{I_1}g|^{2}|\E_{I_2}g|^{4} \lsm \sum_{J \in P_{\nu^{b}}(I_1)}\int_{B}|\E_{J}g|^{2}|\E_{I_2}g|^{4}
\end{align*}
for arbitrary $I_1 \in P_{\nu^{a}}([0, 1])$ and $I_2 \in P_{\nu^{b}}([0, 1])$
such that $d(I_1, I_2) \gtrsim \nu$.
In other words, $M_{a, b}(\delta, \nu) \lsm M_{b, b}(\delta, \nu).$
\end{lemma}
\begin{proof}
It suffices to show that for
each $\Delta' \in P_{\nu^{-b}}(B)$, we have
\begin{align*}
\int_{\Delta'}|\E_{I_1}g|^{2}|\E_{I_2}g|^{4} \lsm \sum_{J \in P_{\nu^{b}}(I_1)}\int_{\Delta'}|\E_{J}g|^{2}|\E_{I_2}g|^{4}.
\end{align*}
Since $I_2$ is an interval of length $\nu^b$, $|\E_{I_2}g|$ is essentially constant on $\Delta'$. Therefore the above reduces
to showing
\begin{align*}
\int_{\Delta'}|\E_{I_1}g|^{2} \lsm \sum_{J \in P_{\nu^{b}}(I_1)}\int_{\Delta'}|\E_{J}g|^{2}
\end{align*}
which since $a < b$ and $I_1$ is of length $\nu^a$ is just an application of $l^2 L^2$ decoupling. This completes the proof of Lemma \ref{mglem2}.
\end{proof}

Inequality \eqref{ineq2_unc} is a consequence of the following ball inflation lemma
which is reminiscent of the ball inflation in the Bourgain-Demeter-Guth proof of Vinogradov's mean value theorem.
The main point of this lemma is to increase the spatial scale so we can apply $l^2 L^2$ decoupling while keeping the frequency scales
constant.
\begin{lemma}[Ball inflation]\label{ef2d_ball}
Let $b \geq 1$ be a positive integer.
Suppose $I_1$ and $I_2$ are intervals of length $\nu^b$ with $d(I_1, I_2) \gtrsim \nu$. Then for any square $\Delta'$ of side length $\nu^{-2b}$,
we have
\begin{align*}
\avg{\Delta \in P_{\nu^{-b}}(\Delta')}\nms{\E_{I_1}g}_{L^{2}_{\#}(\Delta)}^{2}\nms{\E_{I_2}g}_{L^{4}_{\#}(\Delta)}^{4} \lsm \nu^{-1}\nms{\E_{I_1}g}_{L^{2}_{\#}(\Delta')}^{2}\nms{\E_{I_2}g}_{L^{4}_{\#}(\Delta')}^{4}.
\end{align*}
\end{lemma}
\begin{proof}
The uncertainty principle implies that $|\E_{I_1}g|$ and $|\E_{I_2}g|$ are essentially constant on $\Delta$.
Therefore we essentially have
\begin{align*}
\avg{\Delta \in P_{\nu^{-b}}(\Delta')}\nms{\E_{I_1}g}_{L^{2}_{\#}(\Delta)}^{2}\nms{\E_{I_2}g}_{L^{4}_{\#}(\Delta)}^{4} &\sim \frac{1}{|P_{\nu^{-b}}(\Delta')|}\sum_{\Delta \in P_{\nu^{-b}}(\Delta')}\frac{1}{|\Delta|}\int_{\Delta}|\E_{I_1}g|^{2}|\E_{I_2}g|^{4}\\
&= \frac{1}{|\Delta'|}\int_{\Delta'}|\E_{I_1}g|^{2}|\E_{I_2}g|^{4}.
\end{align*}

Cover $\Delta'$ by disjoint rectangles $\{T_1\}$ of size $\nu^{-b} \times \nu^{-2b}$ pointing in the direction
$(-2c_{I_1}, 1)$ where $c_{I_1}$ is the center of $I_!$. Similarly form the collection of $\nu^{-b} \times \nu^{-2b}$
rectangles $\{T_2\}$ corresponding to $I_2$. From the uncertainty principle,
$|\E_{I_1}g| \sim \sum_{T_1}|a_{T_1}|1_{T_1}$ and $|\E_{I_2}g| \sim \sum_{T_2}|a_{T_2}|1_{T_2}$
for some constants $|a_{T_i}|$ depending on $T_i, g$, and $\Delta'$.

Since $I_1$ and $I_2$ are $O(\nu)$-separated, for any two tubes $T_1, T_2$ corresponding to $I_1, I_2$, we have
$|T_1 \cap T_2| \lsm \nu^{-1 - 2b}$.
Therefore
\begin{align*}
\frac{1}{|\Delta'|}\int_{\Delta'}|\E_{I_1}g|^{2}|\E_{I_2}g|^{4} \lsm \nu^{-1}\frac{\nu^{-2b}}{|\Delta'|}\sum_{T_1, T_2} |a_{T_1}|^{2}|a_{T_2}|^{4}.
\end{align*}
Since
\begin{align*}
\nms{\E_{I_1}g}_{L^{2}_{\#}(\Delta')}^{2}\nms{\E_{I_2}g}_{L^{4}_{\#}(\Delta')}^{4} \sim \frac{\nu^{-6b}}{|\Delta'|^2}\sum_{T_1, T_2}|a_{T_1}|^{2}|a_{T_2}|^{4}
\end{align*}
and $|\Delta'| = \nu^{-4b}$, this completes the proof of Lemma \ref{ef2d_ball}.
\end{proof}

We now prove inequality \eqref{ineq2_unc}.
\begin{lemma}\label{mglem1}
Suppose $\delta$ and $\nu$ were such that $\nu^{2b}\delta^{-1} \in \N$. Then
\begin{align*}
\int_{B}|\E_{I_1}g|^{2}|\E_{I_2}g|^{4} \lsm \nu^{-1}\sum_{J \in P_{\nu^{2b}}(I_1)}\int_{B}|\E_{J}g|^{2}|\E_{I_2}g|^{4}
\end{align*}
for arbitrary $I_1 \in P_{\nu^{b}}([0, 1])$ and $I_2 \in P_{\nu^{b}}([0, 1])$
such that $d(I_1, I_2) \gtrsim \nu$.
In other words, $M_{b, b}(\delta, \nu) \lsm \nu^{-1/6}M_{2b, b}(\delta, \nu).$
\end{lemma}
\begin{proof}
This is an application of ball inflation, $l^2 L^2$ decoupling, Bernstein's inequality, and the uncertainty principle.
Since $\nu^{2b}\delta^{-1} \in \N$, $\nu^{b}\delta^{-1} \in \N$ and $\delta \leq \nu^{2b}$.
Fix arbitrary $I_1, I_2 \in P_{\nu^{b}}([0, 1])$.
We have
\begin{align}\label{mgeq2}
\frac{1}{|B|}\int_{B}|\E_{I_1}g|^{2}|\E_{I_2}g|^{4} &= \frac{1}{|B|}\sum_{\Delta \in P_{\nu^{-b}}(B)}\int_{\Delta}|\E_{I_1}g|^{2}|\E_{I_2}g|^{4}\nonumber\\
&\leq \frac{1}{|B|}\sum_{\Delta \in P_{\nu^{-b}}(B)}(\int_{\Delta}|\E_{I_1}g|^{2})\nms{\E_{I_2}g}_{L^{\infty}(\Delta)}^{4}\nonumber\\
&\lsm \frac{1}{|P_{\nu^{-b}}(B)|}\sum_{\Delta \in P_{\nu^{-b}}(B)}(\frac{1}{|\Delta|}\int_{\Delta}|\E_{I_1}g|^{2})\nms{\E_{I_2}g}_{L^{4}_{\#}(\Delta)}^{4}\nonumber\\
&= \avg{\Delta \in P_{\nu^{-b}}(B)}\nms{\E_{I_1}g}_{L^{2}_{\#}(\Delta)}^{2}\nms{\E_{I_2}g}_{L^{4}_{\#}(\Delta)}^{4}
\end{align}
where the second inequality is because of Bernstein's inequality. From ball inflation we know that for each $\Delta' \in P_{\nu^{-2b}}(B)$,
\begin{align*}
\avg{\Delta \in P_{\nu^{-2b}}(\Delta')}\nms{\E_{I_1}g}_{L^{2}_{\#}(\Delta)}^{2}\nms{\E_{I_2}g}_{L^{4}_{\#}(\Delta)}^{4} \lsm \nu^{-1}\nms{\E_{I_1}g}_{L^{2}_{\#}(\Delta')}^{2}\nms{\E_{I_2}g}_{L^{4}_{\#}(\Delta')}^{4}.
\end{align*}
Averaging the above over all $\Delta' \in P_{\nu^{-2b}}(B)$ shows that \eqref{mgeq2} is
\begin{align*}
\lsm \nu^{-1}\avg{\Delta' \in P_{\nu^{-2b}}(B)}\nms{\E_{I_1}g}_{L^{2}_{\#}(\Delta')}^{2}\nms{\E_{I_2}g}_{L^{4}_{\#}(\Delta')}^{4}.
\end{align*}
Since $I_1$ is of length $\nu^{b}$, $l^2 L^2$ decoupling gives that the above is
\begin{align*}
&\lsm \nu^{-1}\sum_{J \in P_{\nu^{2b}}(I_1)}\avg{\Delta' \in P_{\nu^{-2b}}(B)}\nms{\E_{J}g}_{L^{2}_{\#}(\Delta')}^{2}\nms{\E_{I_2}g}_{L^{4}_{\#}(\Delta')}^{4}\\
&= \nu^{-1}\frac{1}{|B|}\sum_{J \in P_{\nu^{2b}}(I_1)}\sum_{\Delta' \in P_{\nu^{-2b}}(B)}\nms{\E_{I_2}g}_{L^{4}(\Delta')}^{4}\nms{\E_{J}g}_{L^{2}_{\#}(\Delta')}^{2}\\
&= \nu^{-1}\frac{1}{|B|}\sum_{J \in P_{\nu^{2b}}(I_1)}\sum_{\Delta' \in P_{\nu^{-2b}}(B)}(\int_{\Delta'}|\E_{I_2}g|^{4})\nms{\E_{J}g}_{L^{2}_{\#}(\Delta')}^{2}.
\end{align*}
Since $|\E_{J}g|$ is essentially constant on $\Delta'$, the uncertainty principle gives that essentially we have
$$(\int_{\Delta'}|\E_{I_2}g|^{4})\nms{\E_{J}g}_{L^{2}_{\#}(\Delta')}^{2} \sim \int_{\Delta'}|\E_{J}g|^{2}|\E_{I_2}g|^{4}.$$
Combining the above two centered equations then completes the proof of Lemma \ref{mglem1}.
\end{proof}
\begin{rem}
The proof of Lemma \ref{mglem1} is reminiscent of our proof of Lemma \ref{abup}. The $\nms{\E_{I_2}g}_{L^{\infty}(\Delta)}$
can be thought as using the trivial bound for $\xi_i$, $i = 2, 3, 5, 6$ to obtain \eqref{f3}. Then we apply some data
about separation, much like in ball inflation here to get large amounts of cancelation.
\end{rem}

\begin{rem}
After the submission of this manuscript, the author along with Shaoming Guo, Po-Lam Yung, and Pavel Zorin-Kranich were able to interpret Wooley's nested efficient congruencing paper \cite{nested}
in terms of decoupling which gave a new rather short proof of $l^2$ decoupling for the moment curve in $\R^k$ \cite{glyzk}. Restricting our paper to $k = 2$ gives a
third proof of Lemma \ref{mglem1} that just uses Plancherel's theorem. In \cite{glyzk} we use the Fourier supported in a neighborhood formulation of decoupling. In what follows we give a heuristic sketch
of the argument using the formulation of decoupling with an extension operator.
See \cite{glyzk} or \cite[Proposition 19]{taonotes} for a rigorous proof. One can also make the argument rigorous using the methods in \cite{guoliyung}.

By affine invariance of the parabola, we may assume that $I_2 = [-\nu^{b}/2, \nu^{b}/2]$ and $I_1 = [d, d + \nu^a]$ where $d \gtrsim \nu$.
From the uncertainty principle, since $I_2 = [-\nu^{b}/2, \nu^{b}/2]$, $|(\E_{I_2}g)(x)|$ is essentially constant on any vertical
$\nu^{-b} \times \nu^{-2b}$ rectangle.
Partition $B$ into vertical $\nu^{-b} \times \nu^{-2b}$ rectangles $\Box$. It suffices to prove that
for each $\Box$, we have
\begin{align*}
\int_{\Box}|\E_{I_1}g|^{2}|\E_{I_2}g|^{4} \lsm \nu^{-1}\sum_{J \in P_{\nu^{2b}}(I_1)}\int_{\Box}|\E_{J}g|^{2}|\E_{I_2}g|^{4}.
\end{align*}
Since $|\E_{I_2}g|$ is essentially constant on $\Box$ and appears on both sides, it suffices to prove
\begin{align*}
\int_{\Box}|\E_{I_1}g|^{2} \lsm \nu^{-1}\sum_{J \in P_{\nu^{2b}}(I_1)}\int_{\Box}|\E_{J}g|^{2}.
\end{align*}
We may assume that $\Box = [0, \nu^{-b}] \times [0, \nu^{-2b}]$. It is enough to prove that
each fixed $x \in [0, \nu^{-b}]$, we have
\begin{align}\label{remarkplan1}
\int_{0}^{\nu^{-2b}}|(\E_{I}g)(x, y)|^{2}\, dy \lsm \nu^{-1}\sum_{J \in P_{\nu^{2b}}(I_1)}\int_{0}^{\nu^{-2b}}|(\E_{J}g)(x, y)|^{2}\, dy.
\end{align}
We claim that this follows from Plancherel's theorem.
Observe that
$|(\E_{I}g)(x, y)| = |\int_{I}g(\xi)e(\xi x)e(\xi^2 y)\, d\xi| = |\int_{d^2}^{(d + \nu^a)^{2}}G_{x}(\eta)e(\eta y)\, d\eta|$
for $G_{x}(\eta) := \frac{1}{2\sqrt{\eta}}g(\sqrt{\eta})e(\sqrt{\eta}x)$.
Let $\mc{P}([d^2, (d + \nu^a)^2])$ be the partition of this interval into intervals $[d^2, (d + \nu^{2b})^2]$, $[(d + \nu^{2b})^2, (d + 2\nu^{2b})^2]$, $[(d + 2\nu^{2b})^{2}, (d + 3\nu^{2b})^2]$, etc.
Let $\psi_{[0, \nu^{-2b}]}$ be a Schwartz function such that $\psi_{[0, \nu^{-2b}]} \geq 1_{[0, \nu^{-2b}]}$ and $\supp(\wh{\psi}_{[0, \nu^{-2b}]}) \subset [-\nu^{2b}/2, \nu^{2b}/2]$.
Then by Plancherel's theorem,
\begin{align*}
\int_{0}^{\nu^{-2b}}|(\E_{I}g)(x, y)|^{2}\, dy &= \int_{0}^{\nu^{-2b}}|\int_{d^2}^{(d + \nu^a)^{2}}G_{x}(\eta)e(\eta y)\, d\eta|^{2}\, dy\\
&= \int_{\R}|\sum_{J \in \mc{P}}\widecheck{G_x 1_{J}}(y)\psi_{[0, \nu^{-2b}]}(y)|^{2}\, dy\\
&= \int_{\R}|\sum_{J \in \mc{P}}G_{x}1_{J} \ast \wh{\psi}_{[0, \nu^{-2b}]}|^{2}.
\end{align*}
Since the $|J| = 2d\nu^{2b} + O(\nu^{2b + a})$, the $G_{x}1_{J} \ast \wh{\psi}_{[0, \nu^{-2b}]}$ have almost pairwise disjoint support, and so the above is (essentially)
\begin{align*}
\lsm \nu^{-1}\sum_{J \in \mc{P}}\int_{0}^{\nu^{-2b}}|\widecheck{G_x 1_{J}}(y)|^{2}\, dy =\nu^{-1}\sum_{J \in P_{\nu^{2b}}(I)}\int_{0}^{\nu^{-2b}}|(\E_{J}g)(x, y)|^{2}\, dy.
\end{align*}
This proves \eqref{remarkplan1} and hence proves Lemma \ref{mglem1}.
\end{rem}

\section{An alternate proof of $D(\delta) \lsm_{\vep} \delta^{-\vep}$}\label{bik}
The ball inflation lemma and our proof of Lemma \ref{mglem1} inspire us to define a new bilinear decoupling
constant that can make our uncertainty principle heuristics from the previous section rigorous.

The left hand side of the definition of $D(\delta)$ in \eqref{decdef} is unweighted, however observe that Corollary \ref{up} implies that
\begin{align}\label{wlhs}
\nms{\E_{[0, 1]}g}_{L^{6}(w_{B})} \lsm D(\delta)(\sum_{J \in P_{\delta}([0, 1])}\nms{\E_{J}g}_{L^{6}(w_{B})}^{2})^{1/2}.
\end{align}
for all $g: [0, 1] \rightarrow \C$ and squares $B$ of side length $\delta^{-2}$.

We will assume that $\delta^{-1} \in \N$ and $\nu \in \N^{-1} \cap (0, 1/100)$.
Let $\mc{M}_{a, b}(\delta, \nu)$ be the best constant such that
\begin{align}\label{bik_const}
\begin{aligned}
\avg{\Delta \in P_{\nu^{-\max(a, b)}}(B)}&\nms{\E_{I}g}_{L^{2}_{\#}(w_{\Delta})}^{2}\nms{\E_{I'}g}_{L^{4}_{\#}(w_{\Delta})}^{4}\\
&\hspace{-0.1in} \leq \mc{M}_{a, b}(\delta, \nu)^{6}(\sum_{J \in P_{\delta}(I)}\nms{\E_{J}g}_{L^{6}_{\#}(w_B)}^{2})(\sum_{J \in P_{\delta}(I')}\nms{\E_{J'}g}_{L^{6}_{\#}(w_B)}^{2})^{2}
\end{aligned}
\end{align}
for all squares $B$ of side length $\delta^{-2}$, $g: [0, 1] \rightarrow \C$ and all intervals $I \in P_{\nu^{a}}([0, 1])$,
$I' \in P_{\nu^{b}}([0, 1])$ with $d(I, I') \geq \nu$.

Suppose $a > b$ (the proof when $a \leq b$ is similar). The uncertainty principle implies that
\begin{align*}
\avg{\Delta \in P_{\nu^{-a}}(B)}\nms{\E_{I_1}g}_{L^{2}_{\#}(\Delta)}^{2}&\nms{\E_{I_2}g}_{L^{4}_{\#}(\Delta)}^{4}\\
& = \frac{1}{|P_{\nu^{-a}}(B)|}\sum_{\Delta \in P_{\nu^{-a}}(B)}(\frac{1}{|\Delta|}\int_{\Delta}|\E_{I_2}g|^{4})\nms{\E_{I_1}g}_{L^{2}_{\#}(\Delta)}^{2}\\
& \sim \frac{1}{|B|}\int_{B}|\E_{I_1}g|^{2}|\E_{I_2}g|^{4}
\end{align*}
where the last $\sim$ is because $|\E_{I_1}g|$ is essentially constant on $\Delta$.
Therefore our bilinear constant $\mc{M}_{a, b}$ is essentially the same as the bilinear constant $M_{a, b}$
we defined in \eqref{mabdef}.

\subsection{Some basic properties}
We now have the weighted rigorous versions of Lemmas \ref{bern_uw} and \ref{l2l2_uw}. Note that we will only need the $L^{\infty}$ version of Lemma \ref{bern_uw}.
\begin{lemma}[Bernstein's inequality]
Let $I$ be an interval of length $1/R$ and $\Delta$ a square of side length $R$.
Then $$\nms{\E_{I}g}_{L^{\infty}(\Delta)} \lsm \nms{\E_{I}g}_{L^{p}_{\#}(w_{\Delta})}.$$
\end{lemma}

\begin{lemma}[$l^2 L^2$ decoupling]
Let $I$ be an interval of length $\geq 1/R$ such that $R|I| \in \N$ and $\Delta$ a square of side length $R$. Then
\begin{align*}
\nms{\E_{I}g}_{L^{2}(w_{\Delta})} \lsm(\sum_{J \in P_{1/R}(I)}\nms{\E_{J}g}_{L^{2}(w_{\Delta})}^{2})^{1/2}.
\end{align*}
\end{lemma}

We now run through the substitutes of Lemmas \ref{mabtriv}-\ref{biv1}.
\begin{lemma}\label{triv}
Suppose $\delta$ and $\nu$ were such that $\nu^{a}\delta^{-1}$, $\nu^{b}\delta^{-1} \in \N$. Then
\begin{align*}
\mc{M}_{a, b}(\delta, \nu) \lsm D(\frac{\delta}{\nu^{a}})^{1/3}D(\frac{\delta}{\nu^b})^{2/3}.
\end{align*}
\end{lemma}
\begin{proof}
Let $I_1 \in P_{\nu^a}([0, 1])$ and $I_2 \in P_{\nu^b}([0, 1])$.
H\"{o}lder's inequality gives that
\begin{align*}
&\avg{\Delta \in P_{\nu^{-\max(a, b)}}(B)}\nms{\E_{I_1}g}_{L^{2}_{\#}(w_{\Delta})}^{2}\nms{\E_{I_2}g}_{L^{4}_{\#}(w_{\Delta})}^{4}\\
&\quad\quad\leq \avg{\Delta \in P_{\nu^{-\max(a, b)}}(B)}\nms{\E_{I_1}g}_{L^{6}_{\#}(w_{\Delta})}^{2}\nms{\E_{I_2}g}_{L^{6}_{\#}(w_{\Delta})}^{4}\\
&\quad\quad\leq (\avg{\Delta \in P_{\nu^{-\max(a, b)}}(B)}\nms{\E_{I_1}g}_{L^{6}_{\#}(w_{\Delta})}^{6})^{1/3}(\avg{\Delta \in P_{\nu^{-\max(a, b)}}(B)}\nms{\E_{I_2}g}_{L^{6}_{\#}(w_{\Delta})}^{6})^{2/3}\\
&\quad\quad \lsm\nms{\E_{I_1}g}_{L^{6}_{\#}(w_{B})}^{2}\nms{\E_{I_2}g}_{L^{6}_{\#}(w_{B})}^{4}
\end{align*}
where the last inequality we have used that $\sum_{\Delta} w_{\Delta} \lsm w_{B}$ (see for example Corollary \ref{part}).
Finally applying \eqref{wlhs} with parabolic rescaling then completes
the proof of Lemma \ref{triv}.
\end{proof}

\begin{lemma}\label{interchange}
Suppose $\nu^{a}\delta^{-1}, \nu^{b}\delta^{-1} \in \N$. Then
\begin{align*}
\mc{M}_{a, b}(\delta, \nu) \lsm \mc{M}_{b, a}(\delta, \nu)^{1/2}D(\frac{\delta}{\nu^b})^{1/2}.
\end{align*}
\end{lemma}
\begin{proof}
Let $I_1 \in P_{\nu^a}([0, 1])$ and $I_2 \in P_{\nu^b}([0, 1])$.
We have
\begin{align*}
&\avg{\Delta \in P_{\nu^{-\max(a, b)}}(B)}\nms{\E_{I_1}g}_{L^{2}_{\#}(w_{\Delta})}^{2}\nms{\E_{I_2}g}_{L^{4}_{\#}(w_{\Delta})}^{4}\\
&\,\,\leq \avg{\Delta \in P_{\nu^{-\max(a, b)}}(B)}\nms{\E_{I_1}g}_{L^{2}_{\#}(w_{\Delta})}^{2}\nms{\E_{I_2}g}_{L^{2}_{\#}(w_{\Delta})}\nms{\E_{I_2}g}_{L^{6}_{\#}(w_{\Delta})}^{3}\\
&\,\,\leq (\avg{\Delta \in P_{\nu^{-\max(a, b)}}(B)}\nms{\E_{I_1}g}_{L^{2}_{\#}(w_{\Delta})}^{4}\nms{\E_{I_2}g}_{L^{2}_{\#}(w_{\Delta})}^{2})^{1/2}(\avg{\Delta \in P_{\nu^{-\max(a, b)}}(B)}\nms{\E_{I_2}g}_{L^{6}_{\#}(w_{\Delta})}^{6})^{1/2}\\
&\,\,\lsm (\avg{\Delta \in P_{\nu^{-\max(a, b)}}(B)}\nms{\E_{I_1}g}_{L^{4}_{\#}(w_{\Delta})}^{4}\nms{\E_{I_2}g}_{L^{2}_{\#}(w_{\Delta})}^{2})^{1/2}\nms{\E_{I_2}g}_{L^{6}_{\#}(w_B)}^{3}
\end{align*}
where the first and second inequalities are because of H\"{o}lder's inequality and the third inequality is an application of
H\"{o}lder's inequality and the estimate $\sum_{\Delta}w_{\Delta} \lsm w_{B}$. Applying
parabolic rescaling and the definition of $\mc{M}_{b, a}$ then completes the proof of Lemma \ref{interchange}.
\end{proof}

\begin{lemma}[Bilinear reduction]\label{bi_red}
Suppose $\delta$ and $\nu$ were such that $\nu\delta^{-1} \in \N$.
Then
$$D(\delta) \lsm D(\frac{\delta}{\nu}) + \nu^{-1}\mc{M}_{1, 1}(\delta, \nu).$$
\end{lemma}
\begin{proof}
The proof is essentially the same as that of Lemma \ref{biv1} except
when analyzing \eqref{holder_bi} in the off-diagonal terms we use
\begin{align*}
\nms{|\E_{I_i}g|^{1/3}|\E_{I_j}g|^{2/3}}_{L^{6}_{\#}(B)}^{6} &= \avg{\Delta \in P_{\nu^{-1}}(B)}\frac{1}{|\Delta|}\int_{\Delta}|\E_{I_i}g|^{2}|\E_{I_j}g|^{4}\\
&\leq \avg{\Delta \in P_{\nu^{-1}}(B)}\nms{\E_{I_i}g}_{L^{2}_{\#}(\Delta)}^{2}\nms{\E_{I_j}g}_{L^{\infty}(\Delta)}^{4}\\
&\lsm \avg{\Delta \in P_{\nu^{-1}}(B)}\nms{\E_{I_i}g}_{L^{2}_{\#}(w_{\Delta})}^{2}\nms{\E_{I_j}g}_{L^{4}_{\#}(w_{\Delta})}^{4}
\end{align*}
where the second inequality we have used Bernstein's inequality.
\end{proof}

\subsection{Ball inflation}
We now prove rigorously the ball inflation lemma we mentioned in the previous section.
\begin{lemma}[Ball inflation]\label{ball_bik}
Let $b \geq 1$ be a positive integer.
Suppose $I_1$ and $I_2$ are $\nu$-separated intervals of length $\nu^b$. Then for any square $\Delta'$ of side length $\nu^{-2b}$,
we have
\begin{align}\label{bmain}
\avg{\Delta \in P_{\nu^{-b}}(\Delta')}\nms{\E_{I_1}g}_{L^{2}_{\#}(w_{\Delta})}^{2}\nms{\E_{I_2}g}_{L^{4}_{\#}(w_{\Delta})}^{4} \lsm \nu^{-1}\nms{\E_{I_1}g}_{L^{2}_{\#}(w_{\Delta'})}^{2}\nms{\E_{I_2}g}_{L^{4}_{\#}(w_{\Delta'})}^{4}.
\end{align}
\end{lemma}
\begin{proof}
Observe that if $c = (c_1, c_2)$, then $(\E_{I}g)(x + c) = (\E_{I}g_{c})(x)$ where $g_{c}(\xi) = g(\xi)e(\xi c_1 + \xi^{2}c_2)$.
Therefore we may assume that $\Delta'$ is centered at the origin.
Fix intervals $I_1$ and $I_2$ intervals of length $\nu^b$ which are $\nu$-separated with centers $c_1$ and $c_2$, respectively.

Cover $\Delta'$ by a set $\mc{T}_1$ of mutually parallel nonoverlapping rectangles $T_1$ of dimensions $\nu^{-b} \times \nu^{-2b}$
with longer side pointing in the direction of $(-2c_1, 1)$ (the normal direction of the piece of parabola above $I_1$).
Note that any such $\nu^{-b} \times \nu^{-2b}$ rectangle outside $4\Delta'$
cannot cover $\Delta'$ itself. Thus we may assume that all rectangles in $\mc{T}_1$ are contained in $4\Delta'$.
Finally let $T_{1}(x)$ be the rectangle in $\mc{T}_1$ containing $x$.
Similarly define $\mc{T}_2$ except this time we use $I_2$.

For $x \in 4\Delta'$, define
\begin{align*}
F_{1}(x) :=
\begin{cases}
\sup_{y \in 2T_{1}(x)}\nms{\E_{I_1}g}_{L^{2}_{\#}(w_{B(y, \nu^{-b})})} & \text{ if } x \in \bigcup_{T_{1} \in \mc{T}_{1}}T_{1}\\
0 & \text{ if } x \in 4\Delta'\bs\bigcup_{T_{1} \in \mc{T}_{1}}T_{1}
\end{cases}
\end{align*}
and
\begin{align*}
F_{2}(x) :=
\begin{cases}
\sup_{y \in 2T_{2}(x)}\nms{\E_{I_2}g}_{L^{4}_{\#}(w_{B(y, \nu^{-b})})} & \text{ if } x \in \bigcup_{T_{2} \in \mc{T}_{2}}T_{2}\\
0 & \text{ if } x \in 4\Delta'\bs\bigcup_{T_{2} \in \mc{T}_{2}}T_{2}.
\end{cases}
\end{align*}
Given a $\Delta \in P_{\nu^{-b}}(\Delta')$, if $x \in \Delta$, then $\Delta \subset 2T_{i}(x)$. This implies that the center of $\Delta$, $c_{\Delta} \in 2T_{i}(x)$ for $x \in \Delta$,
and hence for all $x \in \Delta$,
$$\nms{\E_{I_1}g}_{L^{2}_{\#}(w_{\Delta})} \leq F_{1}(x)$$
and
$$\nms{\E_{I_2}g}_{L^{4}_{\#}(w_{\Delta})} \leq F_{2}(x).$$
Therefore
\begin{align}\label{bins}
\nms{\E_{I_1}g}_{L^{2}_{\#}(w_{\Delta})}^{2}\nms{\E_{I_2}g}_{L^{4}_{\#}(w_{\Delta})}^{4} \leq \frac{1}{|\Delta|}\int_{\Delta}F_{1}(x)^{2}F_{2}(x)^{4}\, dx.
\end{align}
By how $F_i$ is defined, $F_i$ is constant on each $T_i \in \mc{T}_i$. That is, for each $x \in \bigcup_{T_i \in \mc{T}_i}T_{i}$,
$$F_{i}(x) = \sum_{T_{i} \in \mc{T}_i}a_{T_{i}}1_{T_i}(x)$$
for some constants $a_{T_i} \geq 0$.

Thus using \eqref{bins} and that the $T_i$ are disjoint, the left hand side of \eqref{bmain} is bounded above by
\begin{align}\label{bstep2}
\frac{1}{|\Delta'|}\int_{\Delta'}F_{1}(x)^{2}F_{2}(x)^{4}\, dx  &= \frac{1}{|\Delta'|}\sum_{T_1, T_2}a_{T_1}^{2}a_{T_2}^{4}|T_1 \cap T_2|\lsm \nu^{-1}\frac{\nu^{-2b}}{|\Delta'|}\sum_{T_1, T_2}a_{T_{1}}^{2}a_{T_2}^{4}
\end{align}
where in the last inequality we have used that
since $I_1$ and $I_2$ are $\nu$-separated, sine of the angle between $T_1$ and $T_2$ is $\gtrsim \nu$ and hence
$|T_1 \cap T_2| \lsm \nu^{-1 - 2b}$.
Note that
\begin{align*}
\nms{F_{1}}_{L^{2}_{\#}(4\Delta')}^{2} &= \frac{\nu^{-3b}}{|4\Delta'|}\sum_{T_1}a_{T_1}^{2}
\end{align*}
and
\begin{align*}
\nms{F_{2}}_{L^{4}_{\#}(4\Delta')}^{4} &= \frac{\nu^{-3b}}{|4\Delta'|}\sum_{T_2}a_{T_2}^{4}.
\end{align*}
Therefore \eqref{bstep2} is
\begin{align*}
\lsm \nu^{-1}\nms{F_{1}}_{L^{2}_{\#}(4\Delta')}^{2}\nms{F_{2}}_{L^{4}_{\#}(4\Delta')}^{4}.
\end{align*}
Thus we are done if we can prove that $$\nms{F_1}_{L^{2}_{\#}(4\Delta')}^{2} \lsm \nms{\E_{I_1}g}_{L^{2}_{\#}(w_{\Delta'})}^{2}$$
and $$\nms{F_2}_{L^{4}_{\#}(4\Delta')}^{4} \lsm \nms{\E_{I_2}g}_{L^{4}_{\#}(w_{\Delta'})}^{4},$$
but this was exactly what was shown in \cite[Eq. (29)]{sg} (and \cite[Lemma 2.6.3]{thesis} for the same inequality but with explicit constants).
\end{proof}

Our choice of bilinear constant
\eqref{bik_const} makes the rigorous proofs of Lemmas \ref{mglem2} and \ref{mglem1} immediate consequences of ball inflation
and $l^2 L^2$ decoupling.

\begin{lemma}\label{mglem2_rig}
Suppose $1 \leq a < b$ and $\delta$ and $\nu$ were such that $\nu^{b}\delta^{-1} \in \N$.
Then $$\mc{M}_{a, b}(\delta, \nu) \lsm \mc{M}_{b, b}(\delta, \nu).$$
\end{lemma}
\begin{proof}
For arbitrary $I_1 \in P_{\nu^{a}}([0, 1])$ and $I_2 \in P_{\nu^{b}}([0, 1])$ which are $\nu$-separated, it suffices to show that
\begin{align*}
\avg{\Delta \in P_{\nu^{-b}}(B)}\nms{\E_{I_1}g}_{L^{2}_{\#}(w_{\Delta})}^{2}&\nms{\E_{I_2}g}_{L^{4}_{\#}(w_{\Delta})}^{4}\\
& \lsm\sum_{J \in P_{\nu^{b}}(I_1)}\avg{\Delta \in P_{\nu^{-b}}(B)}\nms{\E_{J}g}_{L^{2}_{\#}(w_{\Delta})}^{2}\nms{\E_{I_2}g}_{L^{4}_{\#}(w_{\Delta})}^{4}.
\end{align*}
But this is immediate from $l^2 L^2$ decoupling which completes the proof of Lemma \ref{mglem2_rig}.
\end{proof}

\begin{lemma}\label{mglem1_rig}
Let $b \geq 1$ and suppose $\delta$ and $\nu$ were such that $\nu^{2b}\delta^{-1} \in \N$. Then
$$\mc{M}_{b, b}(\delta, \nu) \lsm \nu^{-1/6}\mc{M}_{2b, b}(\delta, \nu).$$
\end{lemma}
\begin{proof}
For arbitrary $I_1 \in P_{\nu^a}([0, 1])$ and $I_2 \in P_{\nu^b}([0, 1])$ which are $\nu$-separated,
it suffices to prove that
\begin{align*}
\avg{\Delta \in P_{\nu^{-b}}(B)}\nms{\E_{I_1}g}_{L^{2}_{\#}(w_{\Delta})}^{2}&\nms{\E_{I_2}g}_{L^{4}_{\#}(w_{\Delta})}^{4}\\
& \lsm \nu^{-1}\sum_{J \in P_{\nu^{2b}}(I_1)}\avg{\Delta' \in P_{\nu^{-2b}}(B)}\nms{\E_{J}g}_{L^{2}_{\#}(w_{\Delta'})}^{2}\nms{\E_{I_2}g}_{L^{4}_{\#}(w_{\Delta'})}^{4}.
\end{align*}
But this is immediate from ball inflation followed by $l^2 L^2$ decoupling which completes the proof of Lemma \ref{mglem1_rig}.
\end{proof}

Combining Lemmas \ref{interchange}, \ref{mglem2_rig}, and \ref{mglem1_rig} gives the following corollary.
\begin{cor}\label{upgrade}
Suppose $\delta$ and $\nu$ were such that $\nu^{2b}\delta^{-1} \in \N$.
Then
\begin{align*}
\mc{M}_{b, b}(\delta, \nu) \lsm \nu^{-1/6}\mc{M}_{2b, 2b}(\delta, \nu)^{1/2}D(\frac{\delta}{\nu^{b}})^{1/2}.
\end{align*}
\end{cor}

This corollary should be compared to the trivial estimate obtained from Lemma \ref{triv} which implies
$\mc{M}_{b, b}(\delta, \nu) \lsm D(\delta/\nu^b)$.

\subsection{The $O_{\vep}(\delta^{-\vep})$ bound}
We now prove that $D(\delta) \lsm_{\vep} \delta^{-\vep}$. The structure of the argument is essentially the same
as that in Section \ref{fc_iter}.
Repeatedly iterating Corollary \ref{upgrade} and following the same steps in how we derived Lemma \ref{m11iter} and Corollaries \ref{decrec} and \ref{core}
gives the following result.
\begin{lemma}\label{iter1}
Let $N$ be an integer chosen sufficiently large later and let $\delta$ be
such that $\delta^{-1/2^{N}} \in \N$ and $0 < \delta < 100^{-2^{N}}$.
Then
\begin{align*}
D(\delta) \lsm D(\delta^{1 - \frac{1}{2^N}}) + \delta^{-\frac{4}{3 \cdot 2^N}}\prod_{j = 0}^{N - 1}D(\delta^{1 - \frac{1}{2^{N - j}}})^{\frac{1}{2^{j + 1}}}.
\end{align*}
\end{lemma}

Trivial bounds for $D(\delta)$ show that $1 \lsm D(\delta) \lsm \delta^{-1/2}$ for all $\delta \in \N^{-1}$.
Let $\ld$ be the smallest real number such that $D(\delta) \lsm_{\vep} \delta^{-\ld - \vep}$ for all $\delta \in \N^{-1}$.
From the trivial bounds $\ld \in [0, 1/2]$. We claim $\ld = 0$. Suppose $\ld > 0$.

Let $N$ be a sufficiently large integer $\geq \frac{8}{3\ld}$. This implies $$1 + \frac{N}{2} - \frac{4}{3\ld} \geq 1.$$
Lemma \ref{iter1} then implies that for $\delta$ such that $\delta^{-1/2^{N}} \in \N$ and $0 < \delta < 100^{-2^{N}}$, we have
\begin{align*}
D(\delta) \lsm_{\vep} \delta^{-\ld(1 - \frac{1}{2^N}) - \vep} + \delta^{-\ld(1 - \frac{1}{2^N}(1 + \frac{N}{2} - \frac{4}{3\ld})) - \vep} \lsm_{\vep} \delta^{-\ld(1 - \frac{1}{2^N}) - \vep}
\end{align*}
where the last inequality we have applied our choice of $N$.
By almost multiplicity we then have the same estimate for all $\delta \in \N^{-1}$ (with a potentially larger constant depending on $N$).
But this then contradicts minimality of $\ld$. Therefore $\ld = 0$.

\section{Unifying two styles of proof}\label{bds}
We now attempt to unify the Bourgain-Demeter style of decoupling and the style of decoupling mentioned in
the previous section.
In view of Corollary \ref{upgrade}, instead of having two integer parameters $a$ and $b$ we just have one integer parameter.

Let $b$ be an integer $\geq 1$ and choose $s \in [2, 3]$ any real number.
Suppose $\delta \in \N^{-1}$ and $\nu \in \N^{-1} \cap (0, 1/100)$ were such that $\nu^{b}\delta^{-1} \in \N$.
Let $\mb{M}_{b}^{(s)}(\delta, \nu)$ be the best constant such that
\begin{align}\label{mb_bds}
\begin{aligned}
\avg{\Delta \in P_{\nu^{-b}}(B)}&(\sum_{J \in P_{\nu^{b}}(I)}\nms{\E_{J}g}_{L^{2}_{\#}(w_\Delta)}^{2})^{\frac{s}{2}}(\sum_{J' \in P_{\nu^{b}}(I')}\nms{\E_{J'}g}_{L^{2}_{\#}(w_\Delta)}^{2})^{\frac{6 - s}{2}}\\
& \leq \mb{M}_{b}^{(s)}(\delta, \nu)^{6}(\sum_{J \in P_{\delta}(I)}\nms{\E_{J}g}_{L^{6}_{\#}(w_B)}^{2})^{\frac{s}{2}}(\sum_{J' \in P_{\delta}(I')}\nms{\E_{J'}g}_{L^{6}_{\#}(w_B)}^{2})^{\frac{6 - s}{2}}
\end{aligned}
\end{align}
for all squares $B$ of side length $\delta^{-2}$, $g: [0, 1] \rightarrow \C$, and all intervals $I, I' \in P_{\nu}([0, 1])$ which are $\nu$-separated.
Note that left hand side of the definition of $\mb{M}_{b}^{(3)}(\delta, \nu)$ is the same as $A_{6}(q, B^r, q)^{6}$ defined in \cite{sg}
and from the uncertainty principle, $\mb{M}_{1}^{(2)}(\delta, \nu)$ is morally the same as $M_{1, 1}(\delta, \nu)$ defined in \eqref{mabdef}
and $\mc{M}_{1, 1}(\delta, \nu)$ defined in \eqref{bik_const}. The $l^2$ piece in the definition of $\mb{M}_{b}^{(s)}(\delta, \nu)$
is chosen so that we can make the most out of applying $l^2 L^2$ decoupling.

We will use $\mb{M}_{b}^{(s)}$ as our bilinear constant in this section to show that $D(\delta) \lsm_{\vep} \delta^{-\vep}$.
The bilinear constant $\mb{M}_{b}^{(s)}$ obeys essentially the same lemmas as in the previous sections.

\begin{lemma}[cf. Lemmas \ref{mabtriv} and \ref{triv}]\label{mbtriv}
If $\delta$ and $\nu$ were such that $\nu^{b}\delta^{-1} \in \N$, then
\begin{align*}
\mb{M}_{b}^{(s)}(\delta, \nu) \lsm D(\frac{\delta}{\nu^b}).
\end{align*}
\end{lemma}
\begin{proof}
Fix arbitrary $I_1, I_2 \in P_{\nu}([0, 1])$ which are $\nu$-separated.
Moving up from $L^{2}_{\#}$ to $L^{6}_{\#}$ followed by H\"{o}lder's inequality in the average over $\Delta$ bounds the left hand side of \eqref{mb_bds} by
\begin{align*}
(\avg{\Delta \in P_{\nu^{-b}}(B)}(\sum_{J \in P_{\nu^{b}}(I_1)}\nms{\E_{J}g}_{L^{6}_{\#}(w_{\Delta})}^{2})^{\frac{6}{2}})^{\frac{s}{6}}(\avg{\Delta \in P_{\nu^{-b}}(B)}(\sum_{J' \in P_{\nu^{b}}(I_2)}\nms{\E_{J}g}_{L^{6}_{\#}(w_{\Delta})}^{2})^{\frac{6}{2}})^{\frac{6 - s}{6}}.
\end{align*}
Using Minkowski to switch the $l^2$ and $l^6$ sum followed by $\sum_{\Delta}w_{\Delta} \lsm w_{B}$ shows that this is
\begin{align*}
\lsm (\sum_{J \in P_{\nu^{b}}(I_1)}\nms{\E_{J}g}_{L^{6}_{\#}(w_B)}^{2})^{\frac{s}{2}}(\sum_{J' \in P_{\nu^{b}}(I_2)}\nms{\E_{J'}g}_{L^{6}_{\#}(w_B)}^{2})^{\frac{6 - s}{2}}.
\end{align*}
Parabolic rescaling then completes the proof of Lemma \ref{mbtriv}.
\end{proof}

\begin{lemma}[Bilinear reduction, cf. Lemmas \ref{biv1} and \ref{bi_red}]\label{bds_bds}
Suppose $\delta$ and $\nu$ were such that $\nu\delta^{-1} \in \N$. Then
\begin{align*}
D(\delta) \lsm D(\frac{\delta}{\nu}) + \nu^{-1}\mb{M}_{1}^{(s)}(\delta, \nu).
\end{align*}
\end{lemma}
\begin{proof}
Note that the left hand side of the definition of $\mb{M}_{1}^{(s)}(\delta, \nu)$ is
\begin{align*}
\avg{\Delta \in P_{\nu^{-1}}(B)}\nms{\E_{I_1}g}_{L^{2}_{\#}(w_{\Delta})}^{s}\nms{\E_{I_2}g}_{L^{2}_{\#}(w_{\Delta})}^{6 - s}.
\end{align*}
Proceeding as in the proof of Lemmas \ref{biv1} and \ref{bi_red},
for $I_i, I_j \in P_{\nu}([0, 1])$ which are $\nu$-separated, we have
\begin{align}\label{bdseq1}
\nms{|\E_{I_i}g||\E_{I_j}g|}_{L^{3}_{\#}(B)}^{1/2} \leq \nms{|\E_{I_i}g|^{\frac{s}{6}}|\E_{I_j}g|^{1 - \frac{s}{6}}}_{L^{6}_{\#}(B)}^{1/2}\nms{|\E_{I_i}g|^{1-\frac{s}{6}}|\E_{I_j}g|^{\frac{s}{6}}}_{L^{6}_{\#}(B)}^{1/2}.
\end{align}
We have
\begin{align*}
\nms{|\E_{I_i}g|^{\frac{s}{6}}|\E_{I_j}g|^{1 - \frac{s}{6}}}_{L^{6}_{\#}(B)}^{6} &= \avg{\Delta \in P_{\nu^{-1}}(B)}\frac{1}{|\Delta|}\int_{\Delta}|\E_{I_i}g|^{s}|\E_{I_j}g|^{6 - s}\\
&\leq \avg{\Delta \in P_{\nu^{-1}}(B)}\nms{\E_{I_i}g}_{L^{s}_{\#}(\Delta)}^{s}\nms{\E_{I_j}g}_{L^{\infty}(\Delta)}^{6 - s}\\
&\lsm \avg{\Delta \in P_{\nu^{-1}}(B)}\nms{\E_{I_i}g}_{L^{2}_{\#}(w_\Delta)}^{s}\nms{\E_{I_j}g}_{L^{2}_{\#}(w_\Delta)}^{6 - s}
\end{align*}
where the last inequality we have used Bernstein's inequality. Inserting this into \eqref{bdseq1} and applying the definition
of $\mb{M}_{1}^{(s)}(\delta, \nu)$ then completes the proof of Lemma \ref{bds_bds}.
\end{proof}

\begin{lemma}[Ball inflation, cf. Lemma \ref{ball_bik}]\label{ball_bds}
Let $b \geq 1$ be a positive integer.
Suppose $I_1$ and $I_2$ are $\nu$-separated intervals of length $\nu$. Then for any square $\Delta'$ of side length $\nu^{-2b}$
and any $\vep > 0$, we have
\begin{align*}
\avg{\Delta \in P_{\nu^{-b}}(\Delta')}&(\sum_{J \in P_{\nu^{b}}(I_1)}\nms{\E_{J}g}_{L^{s}_{\#}(w_{\Delta})}^{2})^{\frac{s}{2}}(\sum_{J' \in P_{\nu^{b}}(I_2)}\nms{\E_{J'}g}_{L^{6 - s}_{\#}(w_{\Delta})}^{2})^{\frac{6 - s}{2}}\\
&\lsm_{\vep} \nu^{-1 - b\vep}(\sum_{J \in P_{\nu^{b}}(I_1)}\nms{\E_{J}g}_{L^{s}_{\#}(w_{\Delta'})}^{2})^{\frac{s}{2}}(\sum_{J' \in P_{\nu^{b}}(I_2)}\nms{\E_{J'}g}_{L^{6 - s}_{\#}(w_{\Delta'})}^{2})^{\frac{6 - s}{2}}
\end{align*}
\end{lemma}
\begin{proof}
The proof proceeds as in the proof of ball inflation in \cite[Theorem 9.2]{sg} (see also \cite[Section 2.6]{thesis} for more details
and explicit constants in the specific case of the parabola).

From dyadic pigeonholing, since we can lose a $\nu^{-b\vep}$, it suffices to restrict the sum over $J$ and $J'$
to families $\mc{F}_1$ and $\mc{F}_2$ such that for all $J \in \mc{F}_1$, $\nms{\E_{J}g}_{L^{s}_{\#}(w_{\Delta'})}$
are comparable up to a factor of 2 and similarly for all $J' \in \mc{F}_2$. H\"{o}lder's inequality gives
\begin{align*}
&\avg{\Delta \in P_{\nu^{-b}}(\Delta')}(\sum_{J \in\mc{F}_1}\nms{\E_{J}g}_{L^{s}_{\#}(w_{\Delta})}^{2})^{\frac{s}{2}}(\sum_{J' \in \mc{F}_2}\nms{\E_{J'}g}_{L^{6 - s}_{\#}(w_{\Delta})}^{2})^{\frac{6 - s}{2}}\\
&\quad\leq (\# \mc{F}_1)^{\frac{s}{2} - 1}(\# \mc{F}_2)^{\frac{6 - s}{2} - 1}\avg{\Delta \in P_{\nu^{-b}}(\Delta')}(\sum_{J \in\mc{F}_1}\nms{\E_{J}g}_{L^{s}_{\#}(w_{\Delta})}^{s})(\sum_{J' \in \mc{F}_2}\nms{\E_{J'}g}_{L^{6 - s}_{\#}(w_{\Delta})}^{6-s}).
\end{align*}
The proof of Lemma \ref{ball_bik} shows that this is
\begin{align*}
\lsm \nu^{-1}(\# \mc{F}_1)^{\frac{s}{2} - 1}(\# \mc{F}_2)^{\frac{6 - s}{2} - 1}(\sum_{J \in\mc{F}_1}\nms{\E_{J}g}_{L^{s}_{\#}(w_{\Delta'})}^{s})(\sum_{J' \in \mc{F}_2}\nms{\E_{J'}g}_{L^{6 - s}_{\#}(w_{\Delta'})}^{6-s}).
\end{align*}
Since for $J \in \mc{F}_1$ the values of $\nms{\E_{J}g}_{L^{s}_{\#}(w_{\Delta'})}$ are comparable and similarly for $J' \in \mc{F}_2$, the above is
\begin{align*}
\lsm \nu^{-1}(\sum_{J \in\mc{F}_1}\nms{\E_{J}g}_{L^{s}_{\#}(w_{\Delta'})}^{2})^{\frac{s}{2}}(\sum_{J' \in \mc{F}_2}\nms{\E_{J'}g}_{L^{6 - s}_{\#}(w_{\Delta'})}^{2})^{\frac{6 - s}{2}}.
\end{align*}
This completes the proof of Lemma \ref{ball_bds}.
\end{proof}

\begin{lemma}[cf. Corollary \ref{upgrade}]\label{mbup_bds}
Suppose $\delta$ and $\nu$ were such that $\nu^{2b}\delta^{-1} \in \N$. Then for every $\vep > 0$,
\begin{align*}
\mb{M}_{b}^{(s)}(\delta, \nu) \lsm_{\vep} \nu^{-\frac{1}{6}(1 + b\vep)}\mb{M}_{2b}^{(s)}(\delta, \nu)^{1/2}D(\frac{\delta}{\nu^b})^{1/2}.
\end{align*}
\end{lemma}
\begin{proof}
Let $\ta$ and $\vp$ be such that $\frac{\ta}{2} + \frac{1 - \ta}{6} = \frac{1}{s}$ and
$\frac{\vp}{2} + \frac{1 - \vp}{6} = \frac{1}{6 - s}$. Then H\"{o}lder's inequality gives
$\nms{f}_{L^{s}} \leq \nms{f}_{L^{2}}^{\ta}\nms{f}_{L^{6}}^{1 - \ta}$
and $\nms{f}_{L^{6 - s}} \leq \nms{f}_{L^{2}}^{\vp}\nms{f}_{L^{6}}^{1 - \vp}$.

Fix arbitrary $I_1, I_2 \in P_{\nu}([0, 1])$ which are $\nu$-separated.
We have
\begin{align*}
&\avg{\Delta \in P_{\nu^{-b}}(B)}(\sum_{J \in P_{\nu^{b}}(I_1)}\nms{\E_{J}g}_{L^{2}_{\#}(w_{\Delta})}^{2})^{\frac{s}{2}}(\sum_{J' \in P_{\nu^{b}}(I_2)}\nms{\E_{J'}g}_{L^{2}_{\#}(w_{\Delta})}^{2})^{\frac{6 - s}{2}}\\
&\leq \avg{\Delta' \in P_{\nu^{-2b}}(B)}\avg{\Delta \in P_{\nu^{-b}}(\Delta')}(\sum_{J \in P_{\nu^{b}}(I_1)}\nms{\E_{J}g}_{L^{s}_{\#}(w_{\Delta})}^{2})^{\frac{s}{2}}(\sum_{J' \in P_{\nu^{b}}(I_2)}\nms{\E_{J'}g}_{L^{6-s}_{\#}(w_{\Delta})}^{2})^{\frac{6 - s}{2}}\\
&\lsm_{\vep} \nu^{-1 - b\vep}\avg{\Delta' \in P_{\nu^{-2b}}(B)}(\sum_{J \in P_{\nu^{b}}(I_1)}\nms{\E_{J}g}_{L^{s}_{\#}(w_{\Delta'})}^{2})^{\frac{s}{2}}(\sum_{J' \in P_{\nu^{b}}(I_2)}\nms{\E_{J'}g}_{L^{6-s}_{\#}(w_{\Delta'})}^{2})^{\frac{6 - s}{2}}
\end{align*}
where the first inequality is from H\"{o}lder's inequality and the second inequality is from ball inflation.
We now use how $\ta$ and $\vp$ are defined to return to a piece which we control by $l^2 L^2$ decoupling
and a piece which we can control by parabolic rescaling.
H\"{o}lder's inequality (as in the definition of $\ta$ and $\vp$) gives that the average above is bounded by
\begin{align*}
\avg{\Delta' \in P_{\nu^{-2b}}(B)}(\sum_{J \in P_{\nu^{b}}(I_1)}\nms{\E_{J}g}_{L^{2}_{\#}(w_{\Delta'})}^{2\ta}&\nms{\E_{J}g}_{L^{6}_{\#}(w_{\Delta'})}^{2(1 - \ta)})^{\frac{s}{2}}\times\\
&(\sum_{J' \in P_{\nu^{b}}(I_2)}\nms{\E_{J'}g}_{L^{2}_{\#}(w_{\Delta'})}^{2\vp}\nms{\E_{J'}g}_{L^{6}_{\#}(w_{\Delta'})}^{2(1 - \vp)})^{\frac{6 - s}{2}}.
\end{align*}
H\"{o}lder's inequality in the sum over $J$ and $J'$ shows that this is
\begin{align*}
\leq \avg{\Delta' \in P_{\nu^{-2b}}(B)}\bigg(&(\sum_{J \in P_{\nu^{b}}(I_1)}\nms{\E_{J}g}_{L^{2}_{\#}(w_{\Delta'})}^{2})^{\ta}(\sum_{J \in P_{\nu^{b}}(I_1)}\nms{\E_{J}g}_{L^{6}_{\#}(w_{\Delta'})}^{2})^{1 - \ta}\bigg)^{\frac{s}{2}}\times\\
&\bigg((\sum_{J' \in P_{\nu^{b}}(I_2)}\nms{\E_{J'}g}_{L^{2}_{\#}(w_{\Delta'})}^{2})^{\vp}(\sum_{J' \in P_{\nu^{b}}(I_2)}\nms{\E_{J'}g}_{L^{6}_{\#}(w_{\Delta'})}^{2})^{1 - \vp}\bigg)^{\frac{6 - s}{2}}.
\end{align*}
Since $\ta s = 3 - \frac{s}{2}$ and $\vp(6 - s) = \frac{s}{2}$, rearranging the above gives
\begin{align*}
\avg{\Delta' \in P_{\nu^{-2b}}(B)}\bigg(&(\sum_{J \in P_{\nu^{b}}(I_1)}\nms{\E_{J}g}_{L^{2}_{\#}(w_{\Delta'})}^{2})^{\frac{1}{2}(3 - \frac{s}{2})}(\sum_{J' \in P_{\nu^{b}}(I_2)}\nms{\E_{J'}g}_{L^{2}_{\#}(w_{\Delta'})}^{2})^{\frac{1}{2} \cdot \frac{s}{2}}\bigg)\times\\
&\bigg((\sum_{J \in P_{\nu^{b}}(I_1)}\nms{\E_{J}g}_{L^{6}_{\#}(w_{\Delta'})}^{2})^{\frac{1}{2}\cdot 3(\frac{s}{2} - 1)}(\sum_{J' \in P_{\nu^{b}}(I_2)}\nms{\E_{J'}g}_{L^{6}_{\#}(w_{\Delta'})}^{2})^{\frac{1}{2} \cdot 3(2 - \frac{s}{2})}\bigg).
\end{align*}
Applying the Cauchy-Schwarz inequality in the average over $\Delta'$ then bounds the above by
\begin{align}\label{two_term}
\begin{aligned}
&\bigg(\avg{\Delta' \in P_{\nu^{-2b}}(B)}(\sum_{J \in P_{\nu^{b}}(I_1)}\nms{\E_{J}g}_{L^{2}_{\#}(w_{\Delta'})}^{2})^{\frac{6-s}{2}}(\sum_{J' \in P_{\nu^{b}}(I_2)}\nms{\E_{J'}g}_{L^{2}_{\#}(w_{\Delta'})}^{2})^{\frac{s}{2}}\bigg)^{\frac{1}{2}}\times\\
&\bigg(\avg{\Delta' \in P_{\nu^{-2b}}(B)}(\sum_{J \in P_{\nu^{b}}(I_1)}\nms{\E_{J}g}_{L^{6}_{\#}(w_{\Delta'})}^{2})^{\frac{3(s - 2)}{2}}(\sum_{J' \in P_{\nu^{b}}(I_2)}\nms{\E_{J'}g}_{L^{6}_{\#}(w_{\Delta'})}^{2})^{\frac{3(4 - s)}{2}}\bigg)^{\frac{1}{2}}.
\end{aligned}
\end{align}
After $l^2 L^2$ decoupling, the first term in \eqref{two_term} is
\begin{align}\label{m2b}
\lsm \mb{M}_{2b}^{(s)}(\delta, \nu)^{3}(\sum_{J \in P_{\delta}(I_1)}\nms{\E_{J}g}_{L^{6}_{\#}(w_B)}^{2})^{\frac{1}{2}\cdot \frac{6-s}{2}}(\sum_{J' \in P_{\delta}(I_2)}\nms{\E_{J'}g}_{L^{6}_{\#}(w_B)}^{2})^{\frac{1}{2}\cdot \frac{s}{2}}.
\end{align}
H\"{o}lder's inequality in the average over $\Delta'$ bounds the second term in \eqref{two_term} by
\begin{align*}
(\avg{\Delta' \in P_{\nu^{-2b}}(B)}(\sum_{J \in P_{\nu^b}(I_1)}\nms{\E_{J}g}_{L^{6}_{\#}(w_{\Delta'})}^{2})^{\frac{6}{2}})^{\frac{s-2}{4}}(\avg{\Delta' \in P_{\nu^{-2b}}(B)}(\sum_{J \in P_{\nu^b}(I_1)}\nms{\E_{J}g}_{L^{6}_{\#}(w_{\Delta'})}^{2})^{\frac{6}{2}})^{\frac{4-s}{4}}.
\end{align*}
Applying Minkowski to interchange the $l^2$ and $l^6$ norms shows that this is
\begin{align*}
\lsm (\sum_{J \in P_{\nu^b}(I_1)}\nms{\E_{J}g}_{L^{6}_{\#}(w_B)}^{2})^{\frac{3(s - 2)}{4}}(\sum_{J' \in P_{\nu^b}(I_2)}\nms{\E_{J'}g}_{L^{6}_{\#}(w_B)}^{2})^{\frac{3(4-s)}{4}}.
\end{align*}
Parabolic rescaling bounds this by
\begin{align}\label{2nd}
D(\frac{\delta}{\nu^b})^{3}(\sum_{J \in P_{\delta}(I_1)}\nms{\E_{J}g}_{L^{6}_{\#}(w_B)}^{2})^{\frac{1}{2} \cdot\frac{3(s - 2)}{2}}(\sum_{J' \in P_{\delta}(I_2)}\nms{\E_{J'}g}_{L^{6}_{\#}(w_B)}^{2})^{\frac{1}{2} \cdot \frac{3(4 - s)}{2}}.
\end{align}
Combining \eqref{m2b} and \eqref{2nd} then completes the proof of Lemma \ref{mbup_bds}.
\end{proof}

With Lemma \ref{mbup_bds},
following the same steps in how we derived Lemma \ref{m11iter} and Corollaries \ref{decrec} and \ref{core} gives the following.
\begin{lemma}[cf. Corollary \ref{core} and Lemma \ref{iter1}]\label{bds_iterate_est}
Let $N$ be an integer chosen sufficient large later and let $\delta$
be such that $\delta^{-1/2^N} \in \N$ and $0 < \delta < 100^{-2^N}$.
Then
\begin{align*}
D(\delta) \lsm_{\vep} D(\delta^{1 - \frac{1}{2^N}}) + \delta^{-\frac{4}{3\cdot 2^N} - \frac{N\vep}{6\cdot 2^N}}\prod_{j = 0}^{N - 1}D(\delta^{1 - \frac{1}{2^{N - j}}})^{\frac{1}{2^{j + 1}}}.
\end{align*}
\end{lemma}

To finish, we proceed as at the end of the previous section. Let $\ld \in [0, 1/2]$ be the smallest
real such that $D(\delta) \lsm_{\vep} \delta^{-\ld - \vep}$. Suppose $\ld > 0$. Choose $N$ such that $$1 + \frac{N}{2} - \frac{4}{3\ld} \geq 1.$$
Then for $\delta$ such that $\delta^{-1/2^{N}} \in \N$ and $0 < \delta < 100^{-2^{N}}$, Lemma \ref{bds_iterate_est} gives
\begin{align*}
D(\delta) \lsm_{\vep} \delta^{-\ld(1 - \frac{1}{2^N}) - \vep} + \delta^{-\ld(1 - \frac{1}{2^N}(1 + \frac{N}{2} - \frac{4}{3\ld})) - \vep(1 - \frac{1}{2^N}) + \frac{N\vep}{2\cdot 2^N} -\frac{N\vep}{6 \cdot 2^N}} \lsm_{\vep} \delta^{-\ld(1 - \frac{1}{2^N}) - \vep}.
\end{align*}
Almost multiplicativity gives that $D(\delta) \lsm_{N, \vep} \delta^{-\ld(1 - \frac{1}{2^N}) - \vep}$ for all $\delta \in \N^{-1}$, contradicting
the minimality of $\ld$.

\section{Discussion on explicit constants}\label{explicitdis}
A close inspection of the proof of Theorem \ref{ef2d_main} reveals that there are two sources
of explicit constants, one from the various weight functions adapted to $B$ and another
from parabolic rescaling (Lemma \ref{parab_res}).
To keep the paper as self contained as possible, we expand upon where the various
explicit constants come from. Some details will only be briefly sketched as they can be found
with explicit constants in \cite[Sections 2.2-2.4]{thesis}.
The argument we present here for the explicit argument in Lemma \ref{parab_res} is
very slightly simpler and a bit different from that in \cite[Section 2.4]{thesis} but follows the same general philosophy.
We claim no optimality in any explicit constant.

\subsection{Polynomial decaying weights $w_{B, E}$}
For $x \in \R^2$ and $B$ a square centered at $c \in \R^2$ of side length $R$, let
$$w_{B, E}(x) := (1 + \frac{|x - c|}{R})^{-E}.$$
In this notation, the weight function $w_{B}$ defined in \eqref{wbdef} is equal to $w_{B, 100}$.
We include the dependence on $E$ to distinguish between absolute constants and the dependence on the
decay rate of $w_{B, E}$ and later in Lemma \ref{ineq} we will need to use two different $E$.
We also let $D(\delta, E)$ be the same definition as $D(\delta)$ in \eqref{decdef} except $w_B$ on
the right hand side is replaced with $w_{B, E}$.

First we have an easy observation in how $w_{B, E}$ interacts with shear matrices.
\begin{lemma}\label{shear}
Let $S = (\begin{smallmatrix} 1 & a\\0 & 1 \end{smallmatrix})$ with $|a| \leq 2$. Then
$w_{B(0, R), E}(Sx) \leq 3^{E}w_{B(0, R), E}(x)$.
\end{lemma}
\begin{proof}
Since $|a| \leq 2$, $\|S^{-1}\| \leq \sqrt{6}$.
Then $(\frac{1 + |S^{-1}y|}{1 + |y|})^{E} \leq \sqrt{6}^{E}$.
The lemma follows by setting $y = Sx/R$.
\end{proof}

Next, we have the following key property of $w_{B, E}$.
\begin{lemma}\label{wbconvolve}
Let $E \geq 10$. For $0 < R' \leq R$,
\begin{align}\label{wbconveq1}
w_{B(0, R), E} \ast w_{B(0, R'), E} \leq 4^{E}R'^{2}w_{B(0, R), E}.
\end{align}
We also have
\begin{align}\label{wbconveq2}
R^{2}w_{B(0, R), E} \leq 2^{E}1_{B(0, R)} \ast w_{B(0, R), E}.
\end{align}
\end{lemma}
\begin{proof}
We first prove \eqref{wbconveq1}.
We would like to give an upper bound for the expression
\begin{align*}
\frac{1}{R'^{2}}\int_{\R^2}(1 + \frac{|x - y|}{R})^{-E}(1 + \frac{|y|}{R'})^{-E}(1 + \frac{|x|}{R})^{E}\, dy
\end{align*}
depending only on $E$.
A change of variables in $y$ and rescaling $x$ shows that it suffices to give an upper bound for
\begin{align}\label{coeq3}
\int_{\R^2}(1 + |x - \frac{R'}{R}y|)^{-E}(1 + |y|)^{-E}(1 + |x|)^{E}\, dy
\end{align}
depending only on $E$.
If $|x| \leq 1$, then \eqref{coeq3} is
\begin{align*}
\leq 2^{E}\int_{\R^2}(1 + |y|)^{-E}\, dy \leq 2^{E}.
\end{align*}
If $|x| > 1$, then we split \eqref{coeq3} into
\begin{align}\label{conveq2}
(\int_{|x - \frac{R'}{R}y| \leq \frac{|x|}{2}} + \int_{|x - \frac{R'}{R}y| > \frac{|x|}{2}})(1 + |x - \frac{R'}{R}y|)^{-E}(1 + |y|)^{-E}(1 + |x|)^{E}\, dy.
\end{align}
In the case of the first integral in \eqref{conveq2}, $(R'/R)|y| \geq |x| - |x - (R'/R)y| \geq |x|/2$ and hence
\begin{align*}
&\int_{|x - \frac{R'}{R}y| \leq \frac{|x|}{2}}(1 + |x - \frac{R'}{R}y|)^{-E}(1 + |y|)^{-E}(1 + |x|)^{E}\, dy\\
&\leq \frac{(1 + |x|)^{E}}{(1 + (R/R')|x|/2)^{E}}\int_{\R^2}(1 + |x - \frac{R'}{R}y|)^{-E}\, dy \leq (4R'/R)^{E}(R/R')^{2} \leq 4^{E}.
\end{align*}
In the case of the second integral in \eqref{conveq2},
\begin{align*}
\int_{|x - \frac{R'}{R}y| > \frac{|x|}{2}}(1 + |x - \frac{R'}{R}y|)^{-E}(1 + |y|)^{-E}&(1 + |x|)^{E}\, dy\\
& \leq (\frac{1 + |x|}{1 + |x|/2})^{E}\int_{\R^2}(1 + |y|)^{-E}\, dy \leq 2^{E}.
\end{align*}
This then proves \eqref{wbconveq1}.

To prove \eqref{wbconveq2} it suffices to give a lower bound for
\begin{align*}
\frac{1}{R^2}\int_{B(0, R)}(1 + \frac{|x - y|}{R})^{-E}(1 + \frac{|x|}{R})^{E}\, dy
\end{align*}
which depends only on $E$.
As before, rescaling $x$ and a change of variables in $y$ gives that it suffices to give a
lower bound independent of $x$ for
\begin{align*}
\int_{B(0, 1)}(\frac{1 + |x|}{1 + |x - y|})^{E}\, dy \geq (\frac{1 + |x|}{2 + |x|})^{E} \geq 2^{-E}.
\end{align*}
This shows \eqref{wbconveq2} and completes the proof of Lemma \ref{wbconvolve}.
\end{proof}

We have the following immediate corollaries.

\begin{cor}\label{part}
Let $B$ be a square of side length $R$ and let $\mc{B}$ be a disjoint partition of $B$ into squares $\Delta$ with side length $R' < R$. Then
for $E \geq 10$,
\begin{align*}
\sum_{\Delta \in \mc{B}}w_{\Delta, E} \leq 16^{E} w_{B, E}.
\end{align*}
\end{cor}
\begin{proof}
It suffices to prove the case when $B$ is centered at the origin.
Since $\mc{B}$ is a disjoint partition of $B$, $\sum_{\Delta \in \mc{B}} 1_{\Delta} = 1_{B}.$
Convolve both sides by $w_{B(0, R'), E}$. For the left hand side use \eqref{wbconveq2} and for the right hand side
use that $1_{B} \leq 2^{E}w_{B, E}$ and \eqref{wbconveq1}.
\end{proof}
\begin{rem}
The only property we needed in the above proof is that $\sum_{\Delta \in \mc{B}} 1_{\Delta} \leq C1_B$
for some absolute constant $C$. In particular, the same proof will work with finitely overlapping covers and when $R/R' \not\in \N$.
\end{rem}

\begin{cor}\label{convint}
For $1 \leq p < \infty$ and $E \geq 10$,
\begin{align*}
\nms{f}_{L^{p}(w_{B(0, R), E})}^{p} \leq 2^{E}\int_{\R^2}\nms{f}_{L^{p}_{\#}(B(y, R))}^{p}w_{B(0, R), E}(y)\, dy.
\end{align*}
\end{cor}
\begin{proof}
Expanding the right hand side, we see the expression $\frac{2^E}{R^2}1_{B(0, R)} \ast w_{B(0, R), E}$ and then we use \eqref{wbconveq2}.
\end{proof}

\begin{cor}\label{up}
Let $I \subset [0, 1]$ and $\mc{P}$ be a disjoint partition of $I$.
Suppose for some $2 \leq p < \infty$, we have
\begin{align*}
\nms{\E_{I}g}_{L^{p}(B)} \leq C(\sum_{J \in \mc{P}}\nms{\E_{J}g}_{L^{p}(w_{B, E})}^{2})^{1/2}
\end{align*}
for all $g: [0, 1] \rightarrow \C$ and all squares $B$ of side length $R$. Then for each $E \geq 10$, we have
\begin{align*}
\nms{\E_{I}g}_{L^{p}(w_{B, E})} \leq 8^{E/p}C(\sum_{J \in \mc{P}}\nms{\E_{J}g}_{L^{p}(w_{B, E})}^{2})^{1/2}
\end{align*}
for all $g: [0, 1] \rightarrow \C$ and all squares $B$ of side length $R$.
\end{cor}
\begin{proof}
The hypothesis and Corollary \ref{convint} imply that
\begin{align*}
\nms{\E_{I}g}_{L^{p}(w_{B, E})}^{p} \leq 2^{E}R^{-2}C^{p}\int_{\R^2}(\sum_{J \in \mc{P}}\nms{\E_{J}g}_{L^{p}(w_{B(y, R), E})}^{2})^{p/2}w_{B, E}(y)\, dy.
\end{align*}
Since $p \geq 2$, applying Minkowski's inequality shows that this is
$$\leq 2^{E}R^{-2}C^{p}(\sum_{J \in \mc{P}}\nms{\E_{J}g}_{L^{p}(w_{B(0, R), E} \ast w_{B, E})}^{2})^{p/2}.$$
Applying \eqref{wbconveq1} then finishes the proof.
\end{proof}

\subsection{Schwartz weight $\eta_B$}
Given $B = B(c, R)$, in this section we explicitly construct a Schwartz function $\eta_B$ such that
$\eta_B \geq 1_B$ and $\supp(\wh{\eta_B}) \subset B(0, 1/R)$.
It is easy to justify existence of such a function, but we desire explicit quantitative estimates.
A different Schwartz function was constructed in \cite[Section 2.2.2]{thesis} but the construction we provide here is slightly simpler in exposition.

\begin{lemma}\label{1dconst}
Fix $E \geq 100$.
There exists a Schwartz function $\psi$ on $\R$ such that $\psi \geq 1_{[-1/2, 1/2]}$, $\supp(\wh{\psi}) \subset [-1/2, 1/2]$
and
$$|\psi(x)| \leq \frac{2^{E}E^{2E}}{(1 + |x|)^{E}}.$$
\end{lemma}
\begin{proof}
For $\xi \in \R$, let $H_{a}(\xi) := \frac{1}{a}1_{[0, a]}(\xi)$ and define the sequence $a_{j} := \frac{3/\pi^2}{(j + 1)^2}$
for $j \geq 0$. From Theorem 1.3.5 of \cite{hormander}, the function $$U(\xi) := \lim_{k \rightarrow \infty} (H_{a_0} \ast \cdots \ast H_{a_k})(\xi)$$
is a smooth function supported on $[0, 1/2]$ obeying the bounds
$$|U^{(k)}(\xi)| \leq \frac{2^k}{\prod_{j = 0}^{k}a_{j}} = \frac{\pi^2}{3}(\frac{2\pi^2}{3})^{k}(k + 1)!^{2}$$
for $\xi \in [0, 1/2]$.
Observe that $\wc{H_{a}}(x) = e^{\pi i xa}\frac{\sin (\pi xa)}{\pi x a}$
and hence
$$\wc{U}(x) = \prod_{j = 0}^{\infty}e^{\pi i x a_{j}}\frac{\sin (\pi xa_{j})}{\pi x a_{j}}.$$
Let
\begin{align*}
\psi(x) := 2|\wc{U}(x)|^{2} = 2\prod_{j = 0}^{\infty}(\frac{\sin (\pi xa_{j})}{\pi x a_{j}})^{2}.
\end{align*}
Note that $\psi \geq 0$ and for $|x| \leq 1/2$, $\psi(x) \geq \psi(1/2) \geq 1$. Expanding $|\wc{U}|^{2} = \wc{U}\ov{\wc{U}}$, $\wh{\psi} = 2(U \ast \wt{U})$
where $\wt{U}(y) := -\ov{U(-y)}$. Since $U$ is supported on $[0, 1/2]$ and $\wt{U}$ is supported on $[-1/2, 0]$, then $\wh{\psi}$ is supported in $[-1/2, 1/2]$.

Finally we derive some bounds on the decay of $\psi$. The support of $U$ and integration by parts gives that for any $j \geq 1$,
\begin{align*}
|\wc{U}(x)| \leq \frac{1}{(2\pi |x|)^{j}}\nms{U^{(j)}}_{L^{1}([0, 1/2])} \leq \frac{\pi^2}{6}(\frac{\pi}{3})^{j}(j + 1)^{2}\frac{j^{2j}}{|x|^{j}}.
\end{align*}
and hence applying this bound to $j = \lceil E/2 \rceil$ gives
\begin{align*}
|\psi(x)| \leq 2(\frac{\pi^2}{6})^{2}(\frac{\pi}{3})^{2(\frac{E}{2} + 1)}(\frac{E}{2} + 2)^{4}\frac{(3E/5)^{2E + 4}}{|x|^{E}} \leq \frac{\pi^6}{1250}(\frac{3\pi}{25})^{E}E^{8}\frac{E^{2E}}{|x|^E} \leq \frac{E^{2E}}{|x|^E}.
\end{align*}
For $|x| \geq 1$, then $(1 + |x|)^{E}|\psi(x)| \leq 2^{E}E^{2E}$
and for $|x| \leq 1$,
$(1 + |x|)^{E}|\psi(x)| \leq 2 \cdot 2^{E}$.
This completes the proof of Lemma \ref{1dconst}.
\end{proof}

\begin{cor}\label{etaweightwb}
For $x = (x_1, x_2)$, let $\eta(x) := \psi(x_1)\psi(x_2)$.
Fix a square $B = B(c, R)$ of side length $R$. The function $\eta_{B}(x) = \eta(\frac{x - c}{R})$ satisfies
$\eta_{B} \geq 1_{B}$, $\supp(\wh{\eta_B}) \subset B(0, 1/R)$, and for any $E \geq 100$,
$$\eta_{B}(x) \leq 2^{2E}E^{4E} w_{B, E}(x).$$
\end{cor}

\subsection{Explicit proof of Lemma \ref{parab_res}}\label{parab_res_const_sec}
We now discuss how we arrived at the explicit constant $10^{16000}$ in Lemma \ref{parab_res}.
The argument we present here is slightly simpler than the one for the last centered equation in \cite[Page 58]{thesis},
but the argument is essentially the same.

We will need a few more different decoupling constants.
First we have a global decoupling constant.
\begin{defn}
For $J \in P_{\delta}([0, 1])$, let
\begin{align*}
\ta_{J, \delta} := \{(s, s^2 + t): s \in J, |t| \leq \delta^2\}
\end{align*}
and $\Theta_{\delta} := \bigcup_{J \in P_{\delta}([0, 1])}\ta_{J, \delta} = \ta_{[0, 1], \delta}$.
Let $D^{global}(\delta)$ be the best constant such that
\begin{align*}
\nms{f}_{L^{6}(\R^2)} \leq D^{global}(\delta)(\sum_{J \in P_{\delta}([0, 1])}\nms{f_{\ta_{J, \delta}}}_{L^{6}(\R^2)}^{2})^{1/2}
\end{align*}
for all $f$ with Fourier support in $\Theta_{\delta}$.
\end{defn}
This decoupling constant is the easiest to use when wanting to prove various functional
properties of the decoupling constant like monotonicity and parabolic rescaling.

\begin{lemma}[Parabolic rescaling for $D^{global}(\delta)$]\label{globalparab}
Let $I \subset [0, 1]$ be an interval of length $\sigma$ such that $0 < \delta < \sigma < 1$ and
$\sigma, \delta, \delta/\sigma \in \N^{-1}$. Then
\begin{align*}
\nms{f_{\ta_{I, \delta}}}_{L^{6}(\R^2)} \leq D^{global}(\frac{\delta}{\sigma})(\sum_{J \in P_{\delta}(I)}\nms{f_{\ta_{J, \delta}}}_{L^{6}(\R^2)}^{2})^{1/2}.
\end{align*}
\end{lemma}
\begin{proof}
Writing $I = [a, a + \sigma]$, we have
\begin{align*}
|(f_{\ta_{I, \delta}})(x)| &= |\int_{a}^{a + \sigma}\int_{-\delta^2}^{\delta^2}\wh{f}(s, s^2 + t)e(sx_1 + s^{2}x_2)e(tx_2)\, dt\, ds|\\
&= \sigma^{3}|\int_{0}^{1}\int_{-(\frac{\delta}{\sigma})^2}^{(\frac{\delta}{\sigma})^2}\wh{f}(\sigma s' + a, (\sigma s' + a)^{2} + \sigma^{2} t')e(s'(\sigma x_1 + 2a\sigma x_2)\\
&\hspace{2.5in}+ (s'^{2} + t')(\sigma^{2}x_2))\, dt'\, ds'|.
\end{align*}
Also observe that
\begin{align*}
\wh{f}(\sigma s' + a, (\sigma s' + a)^{2} + \sigma^{2} t') = \sigma^{-3}\wh{(F \circ L_{\sigma, a}^{-1})}(s', s'^{2} + t')
\end{align*}
where $F(x) = f(x)e^{-2\pi i x\cdot (a, a^2)}$ and
$$L_{\sigma, a} := \begin{pmatrix} \sigma & 2a\sigma\\0 & \sigma^2\end{pmatrix}.$$
Thus
\begin{align*}
|(f_{\ta_{I, \delta}})(x)| = |(F \circ L_{\sigma, a}^{-1})_{\ta_{[0, 1], (\delta/\sigma)^{2}}}(L_{\sigma, a}x)|.
\end{align*}
Then we apply the definition of $D^{global}(\delta/\sigma)$ and reverse the change of variables which completes the proof of Lemma \ref{globalparab}.
\end{proof}

Having proven Lemma \ref{globalparab} we are now almost done, essentially we just need to apply $f = \eta_{B}\E_{I}g$ to the above lemma and
use that $D^{global}(\delta) \lsm_{E} D(\delta, E)$ (as mentioned in \cite[Remark 5.2]{bd} and essentially follows as a corollary from \cite[Theorem 5.1]{sg}).
There are two small but fixable problems with this argument. The first is that $\eta_{B}\E_{I}g$ is has Fourier support in a region slightly larger than $\ta_{I, \delta}$
and so $(\eta_{B}\E_{I}g)_{\ta_{I, \delta}}$ is not necessarily equal to $\eta_{B}\E_{I}g$ and so
we will instead apply Lemma \ref{globalparab} to $f = \eta_{B}\E_{[a + \delta, a + \sigma - \delta]}g$.
The second is that the $\lsm_{E}$ in the estimate $D^{global}(\delta) \lsm_{E} D(\delta, E)$ is not made explicit.

\begin{rem}
To avoid the use of any equivalence of decoupling constants, one can instead just use Lemma \ref{globalparab} and suitably modify the Lemmas \ref{almostmult}-\ref{abup}.
This is the approach taken in \cite{guoliyung, glyzk} and in Tao's 247B notes on decoupling \cite{taonotes} (whose proof of parabola decoupling is based
off the argument in this paper).
We don't take this point of view here since it somewhat obscures the connection between efficient congruencing and decoupling, in particular when
comparing the proof of Lemma \ref{abup} with \cite[Lemma 4.4]{pierce}.
\end{rem}

Some equivalences between various decoupling constants were made quantitative by the author in \cite[Proposition 2.3.11]{thesis}
and we will use this result here as a black box (the proof is quite similar to that of \cite[Theorem 5.1]{sg}). To state the
relevant part of \cite[Proposition 2.3.11]{thesis} used here, we define two more decoupling constants.

\begin{defn}
For $J = [n_J \delta, (n_J + 1)\delta] \in P_{\delta}([0, 1])$, let
\begin{align*}
\ta_{J, \delta}' := \{(s, L_{J}(s) + t): n_J \delta \leq s \leq (n_J + 1)\delta, |t| \leq 5\delta^{2}\}
\end{align*}
where $L_{J}(s) := (2n_J + 1)\delta s - n_{J}(n_{J} + 1)\delta^{2}$ and $0 \leq n_{J} \leq \delta^{-1} - 1$.
Here $\ta_{J, \delta}'$ is a parallelogram that has height $10\delta^{2}$ and has base parallel to the
straight line connecting $(n_{J}\delta, n_{J}^{2}\delta^{2})$ and $((n_{J} + 1)\delta, (n_{J} + 1)^{2}\delta^{2})$.

Let $\Theta'_{\delta} := \bigcup_{J \in P_{\delta}([0, 1])}\ta_{J, \delta}'$.
Let $D_{par}^{global}(\delta)$ be the best constant such that
\begin{align*}
\nms{f}_{L^{6}(\R^2)} \leq D_{par}^{global}(\delta)(\sum_{J \in P_{\delta}([0, 1])}\nms{f_{\ta_{J, \delta}'}}_{L^{6}(\R^2)}^{2})^{1/2}
\end{align*}
for all $f$ with Fourier support in $\Theta'_{\delta}$.

Let $D_{par}^{local}(\delta, E)$ be the best constant such that
\begin{align*}
\nms{f}_{L^{6}(B)} \leq D_{par}^{local}(\delta, E)(\sum_{J \in P_{\delta}([0, 1])}\nms{f_{\ta_{J, \delta}'}}_{L^{6}(w_{B, E})}^{2})^{1/2}
\end{align*}
for all $f$ with Fourier support in $\Theta'_{\delta}$ and all squares $B$ of side length $\delta^{-2}$.
\end{defn}

\begin{lemma}\label{ineq}
For $E \geq 100$, $$D^{global}(\delta) \leq E^{64E}D(\delta, E).$$
\end{lemma}
\begin{proof}
This lemma is the consequence of the following string of inequalities:
\begin{align}\label{inlem1}
D^{global}(\delta) \leq D_{par}^{global}(\delta) \leq 2^{E}D_{par}^{local}(\delta, \lfloor \frac{E - 7}{2}\rfloor) \leq 2^{E}E^{63E}D(\delta, E)
\end{align}
from which the lemma immediately follows.
The third inequality in \eqref{inlem1} is the last inequality in the statement of \cite[Proposition 2.3.11]{thesis} (written as $E^{7E}\wh{D}_{p, G}(\delta) \leq E^{70E}D_{p, E}(\delta)$).

For each $J \in P_{\delta}([0, 1])$, we have $\ta_{J, \delta} \subset \ta_{J', \delta}$ and hence the first inequality follows.
To prove the second inequality, we let $f$ be a function which is Fourier supported in $\Theta'_{\delta}$. Partition $\R^2$ into squares $B$
of side length $\delta^{-2}$. Then
\begin{align*}
\nms{f}_{L^{6}(\R^2)} &= (\sum_{B}\nms{f}_{L^{6}(B)}^{6})^{1/6}\\
 &\leq D_{par}^{local}(\delta,  \lfloor \frac{E - 7}{2}\rfloor)(\sum_{B}(\sum_{J \in P_{\delta}([0, 1])}\nms{f_{\ta_{J}'}}_{L^{6}(w_{B,  \lfloor \frac{E - 7}{2}\rfloor})}^{2})^{3})^{1/6}\\
&\leq D_{par}^{local}(\delta,  \lfloor \frac{E - 7}{2}\rfloor)(\sum_{J \in P_{\delta}([0, 1])}\nms{f_{\ta_{J}'}}_{L^{6}(\sum_{B} w_{B, \lfloor \frac{E - 7}{2}\rfloor})}^{2})^{1/2}
\end{align*}
where the last inequality is by Minkowski's inequality.
The proof of Corollary \ref{part} (and \eqref{wbconveq2}) shows that $\sum_{B} w_{B,  \lfloor \frac{E - 7}{2}\rfloor} \leq 2^{\lfloor \frac{E - 7}{2}\rfloor}\int_{\R^2}(1 + |x|)^{-\lfloor \frac{E - 7}{2}\rfloor}\, dx \leq 2^{E}$.
This completes the proof of the second inequality and the proof of Lemma \ref{ineq}.
\end{proof}

\begin{rem}
The proof of the last inequality in \eqref{inlem1} is very similar to the proof of Theorem 5.1 of \cite{sg} except all the estimates are made explicit and quantitative.
We illustrate heuristically the main idea of \cite[Theorem 5.1]{sg} (and hence the last inequality in \eqref{inlem1}). We will ignore any weight functions $w_{B, E}$.

Define a decoupling constant $D^{local}(\delta)$ that is the same as the definition of $D_{par}^{local}(\delta)$ except that
$\ta_{J, \delta}'$ are replaced with $\ta_{J, \delta}$. Then we show heuristically why we expect $D^{local}(\delta) \lsm D(\delta)$.
Since $f$ is Fourier supported in $\{(s, s^2 + t): s \in [0, 1], |t| \leq \delta^2\}$.
Ignoring any $E$ dependence and weights $w_{B, E}$, we want to argue that
\begin{align*}
\nms{f}_{L^{6}(B)} \lsm D(\delta)(\sum_{J \in P_{\delta}([0, 1])}\nms{f_{\ta_J}}_{L^{6}(B)}^{2})^{1/2}
\end{align*}
for all squares $B$ of side length $\delta^{-2}$. Without loss of generality we may assume that $B$ is centered at the origin.
Using the Fourier support of $f$, we can write
\begin{align*}
f(x) = \int_{0}^{1}(\int_{-\delta^{2}}^{\delta^{2}}\wh{f}(s, s^2 + t)e(tx_2)\, dt)\, e(sx_1 + s^{2}x_2)\, ds.
\end{align*}
For $x \in B$ and since $B$ is centered at the origin and $|t| \leq \delta^{2}$, $tx_2$ does not oscillate much and so
we will pretend that $e(tx_2) \approx 1$. Then the above is essentially equal to $\E_{[0, 1]}F$ where
$F(s) := \int_{-\delta^2}^{\delta^2}\wh{f}(s, s^2 + t)\, dt.$
We are done after applying the definition of $D(\delta)$, undoing the definition of $\E_{J}F$, and adding the $e(tx_2)$ back in.
The rigorous proof in \cite[Theorem 5.1]{sg} involves expanding $e(tx_2)$
as a Taylor series and to each term in the Taylor series we create an $F_{j}$ and show that
$\nms{\E_{J}F_{j}}_{L^{6}(B)} \lsm \exp(O(j)) \nms{f_{\ta_{J}}}_{L^{6}(B)}$.
\end{rem}

Finally we will need a quantitative result about Fourier restriction to a parallelogram.
\begin{lemma}\label{hilbert}
For each $J \in P_{\delta}([0, 1])$ and $2 \leq p < \infty$, $\nms{f_{\ta'_{J, \delta}}}_{p} \leq ( \frac{1}{2} + \frac{1}{2}\cot(\frac{\pi}{2p}))^{4} \nms{f}_{p}$.
\end{lemma}
\begin{proof}
Let $S$ denote the operator defined by $\wh{S g}(\eta) = \wh{g}(\eta)1_{[0, \infty)}(\eta)$ for $\eta \in \R$.
If $H$ denotes the Hilbert transform, observe that $\wh{f}(\eta) + i\wh{Hf}(\eta) = 2\wh{S f}(\eta)$
almost everywhere.
Since $2 \leq p < \infty$, $\nms{H}_{p \rightarrow p} \leq \cot(\frac{\pi}{2p})$ and so $\nms{S}_{p \rightarrow p} \leq \frac{1}{2} + \frac{1}{2}\cot(\frac{\pi}{2p})$.

Let $R$ denote the operator defined by $\wh{R f}(\xi) = \wh{f}(\xi)1_{\ta'_{J, \delta}}(\xi)$ for $\xi = (\xi_1, \xi_2) \in \R^2$.
Each $\ta'_{J, \delta}$ is the intersection of four half planes in $\R^2$. The operator norm of Fourier restriction to a half plane
is the same as the operator norm for Fourier restriction to the half plane $[0, \infty) \times \R$.
By Fubini's Theorem, this operator norm is bounded above by the operator norm for $S$. Therefore
$\nms{R}_{p \rightarrow p} \leq \nms{S}_{p \rightarrow p}^{4} \leq (\frac{1}{2} + \frac{1}{2}\cot(\frac{\pi}{2p}))^{4}$
which completes the proof of Lemma \ref{hilbert}.
\end{proof}

We now have all the ingredients to give an explicit proof of Lemma \ref{parab_res}.
\begin{proof}
Write $I = [a, a + \sigma]$. We have
\begin{align}\label{quanteq0}
\nms{\E_{I}g}_{L^{6}(B)} \leq \nms{\E_{[a, a + \delta]}g}_{L^{6}(B)} + \nms{\eta_{B}\E_{[a + \delta, a + \sigma - \delta]}g}_{L^{6}(\R^2)} + \nms{\E_{[a + \sigma - \delta, a + \sigma]}g}_{L^{6}(B)}.
\end{align}
The Fourier transform of $\eta_{B}\E_{[a + \delta, a + \sigma - \delta]}g$ is supported in $\ta_{I, \delta}$
and so combining Lemma \ref{globalparab} with Lemma \ref{ineq} shows that
\begin{align}\label{quanteq1}
\nms{\eta_{B}\E_{[a + \delta, a + \sigma - \delta]}g}_{L^{6}(\R^2)} &\leq E^{64E}D(\frac{\delta}{\sigma}, E)(\sum_{J \in P_{\delta}(I)}\nms{(\eta_{B}\E_{[a + \delta, a + \sigma - \delta]}g)_{\ta_{J, \delta}}}_{L^{6}(\R^2)}^{2})^{1/2}.
\end{align}
Observe that
\begin{align*}
(\eta_{B}&\E_{[a + \delta, a + \sigma - \delta]}g)_{\ta_{J, \delta}}\\
&=
\begin{cases}
(\eta_{B}\E_{J_r}g)_{\ta_{J, \delta}} & \text{ if } J = [a, a + \delta]\\
(\eta_{B}\E_{J}g + \eta_{B}\E_{J_r}g)_{\ta_{J, \delta}} & \text{ if } J = [a + \delta, a + 2\delta]\\
(\eta_B \E_{J_{\ell}}g + \eta_B \E_{J}g + \eta_B \E_{J_r}g)_{\ta_{J, \delta}} & \text{ if } J \in P_{\delta}([a + 2\delta, a + \sigma - 2\delta])\\
(\eta_{B}\E_{J_{\ell}}g + \eta_{B}\E_{J}g)_{\ta_{J, \delta}} & \text{ if } J = [a + \sigma - 2\delta, a + \sigma - \delta]\\
(\eta_{B}\E_{J_{\ell}}g)_{\ta_{J, \delta}} & \text{ if } J = [a + \sigma - \delta, a + \sigma]
\end{cases}
\end{align*}
where $J_{\ell}$ and $J_{r}$ denote the intervals to the left and right of $J$.
Therefore for $J \in P_{\delta}([a, a + \sigma])$,
\begin{align*}
\nms{(\eta_{B}\E_{[a + \delta, a + \sigma - \delta]}g)_{\ta_{J,\delta}}}_{L^{6}(\R^2)}^{2} &\leq (\sum_{\st{J' \in \{J_{\ell}, J, J_{r}\}\\J' \subset [a + \delta, a + \sigma - \delta]}}\nms{(\eta_{B}\E_{J'}g)_{\ta_{J, \delta}}}_{L^{6}(\R^2)})^{2}\\
&= (\sum_{\st{J' \in \{J_{\ell}, J, J_{r}\}\\J' \subset [a + \delta, a + \sigma - \delta]}}\nms{(\eta_{B}\E_{J'}g)_{\ta_{J, \delta}'}}_{L^{6}(\R^2)})^{2}\\
&\leq 32^{2}(\sum_{\st{J' \in \{J_{\ell}, J, J_{r}\}\\J' \subset [a + \delta, a + \sigma - \delta]}}\nms{\E_{J'}g}_{L^{6}(\eta_B)})^{2}\\
&\leq 3 \cdot 32^{2} (2^{2E}E^{4E})^{1/3}\sum_{\st{J' \in \{J_{\ell}, J, J_{r}\}\\J' \subset [a + \delta, a + \sigma - \delta]}}\nms{\E_{J'}g}_{L^{6}(w_{B, E})}^{2}
\end{align*}
where in the first equality we have used that $\ta_{J} \subset \ta_{J}'$, in the second inequality we have used Lemma \ref{hilbert}, and in the third inequality we have
used Lemma \ref{etaweightwb}.
Inserting this into \eqref{quanteq1} shows that the right hand side is
\begin{align*}
&\leq E^{64E}(3 \cdot 32^{2} (2^{2E}E^{4E})^{1/3})^{1/2}D(\frac{\delta}{\sigma}, E)(\sum_{J \in P_{\delta}(I)}\sum_{\st{J' \in \{J_{\ell}, J, J_r\}\\J' \subset [a + \delta, a + \sigma - \delta]}}\nms{\E_{J'}g}_{L^{6}(w_{B, E})}^{2})^{1/2}\\
&\leq E^{65E}D(\frac{\delta}{\sigma}, E)(\sum_{J \in P_{\delta}(I)}\nms{\E_{J}g}_{L^{6}(w_{B, E})}^{2})^{1/2}.
\end{align*}
Inserting this into \eqref{quanteq0} and using that $1_{B} \leq 2^{E}w_{B, E}$ then shows that
\begin{align*}
\nms{\E_{I}g}_{L^{6}(B)} \leq E^{80E} D(\frac{\delta}{\sigma}, E)(\sum_{J \in P_{\delta}(I)}\nms{\E_{J}g}_{L^{6}(w_{B, E})}^{2})^{1/2}.
\end{align*}
Taking $E = 100$ completes the proof of Lemma \ref{parab_res} with explicit constants.
\end{proof}

\bibliographystyle{amsplain}
\bibliography{efdec_v3}
\end{document}